\documentclass[10.5pt, a4paper]{article}
\usepackage{fullpage}
\usepackage{listings}
\usepackage{latexsym, amsfonts, amssymb, amsmath, amsthm, mathrsfs, mathtools, setspace, graphics, graphicx, bbm, float, bigints}
\usepackage{enumerate}
\usepackage{caption}
\usepackage{indentfirst}
\usepackage{framed}
\usepackage{yhmath}
\usepackage[nodayofweek,level]{datetime}
\usepackage{overpic}

\RequirePackage{tikz}
\usetikzlibrary{calc,trees,positioning,arrows,chains,shapes.geometric,%
    decorations.pathreplacing,decorations.pathmorphing,shapes,%
    matrix,shapes.symbols}
\allowdisplaybreaks    

\usepackage{geometry}
\title{Monotone convex order for the McKean-Vlasov processes}

\usepackage{hyperref}
\newcommand{\footremember}[2]{
   \footnote{#2}
    \newcounter{#1}
    \setcounter{#1}{\value{footnote}}
}

\author{%
 Yating Liu\footremember{a}{\small CEREMADE, CNRS, UMR 7534, Université Paris-Dauphine, PSL University, 75016 Paris, France, \texttt{liu@ceremade.dauphine.fr}.}%
  \and Gilles Pag{\`e}s\footremember{b}{\small Sorbonne Université, CNRS, Laboratoire de Probabilités, Statistique et Modélisation (LPSM),
75252 Paris, France, \texttt{gilles.pages@sorbonne-universite.fr}.}%
  }

\begin{document}
\maketitle

\numberwithin{equation}{section}
\newtheorem{thm}{Theorem}[section]
\newtheorem{lem}{Lemma}[section]
\newtheorem{prop}{Proposition}[section]
\newtheorem{cor}{Corollary}[section]
\newtheorem{defn}{Definition}[section]
\theoremstyle{remark}
\newtheorem{rem}{Remark}[section]

\newcommand{\comp}[1]{{#1}^{\mathsf{c}}} 

\newcommand{\widesim}[2][1.5]{
  \mathrel{\overset{#2}{\scalebox{#1}[1]{$\sim$}}}
}

\newcommand{\RRD}{\mathbb{R}^{d}}
\newcommand{\vertiii}[1]{{\left\vert\kern-0.25ex\left\vert\kern-0.25ex\left\vert #1 
    \right\vert\kern-0.25ex\right\vert\kern-0.25ex\right\vert}}
\newcommand{\vertii}[1]{\left\Vert #1\right\Vert}
\newcommand{\Proj}{\mathrm{Proj}}

\newcommand{\tabincell}[2]{\begin{tabular}{@{}#1@{}}#2\end{tabular}}
\newcommand*\circled[1]{\tikz[baseline=(char.base)]{
            \node[shape=circle,draw,inner sep=2pt] (char) {#1};}}

\newcommand\independent{\protect\mathpalette{\protect\independenT}{\perp}}
\def\independenT#1#2{\mathrel{\rlap{$#1#2$}\mkern2mu{#1#2}}}

\newcommand{\PPC}{\mathcal{P}_{p}\big(\mathcal{C}([0, T], \mathbb{R}^{d})\big)}
\newcommand{\CPP}{\mathcal{C}\big([0, T], \mathcal{P}_{p}(\mathbb{R}^{d})\big)}
\newcommand{\CRD}{\mathcal{C}([0, T], \mathbb{R}^{d})}
\newcommand{\PPRD}{\mathcal{P}_{p}(\mathbb{R}^{d})}

\newcommand{\PP}{\mathbb{P}}
\newcommand{\RD}{\mathbb{R}^{d}}
\newcommand{\RR}{\mathbb{R}}
\newcommand{\PRD}{\mathcal{P}(\mathbb{R}^{d})}
\newcommand{\MDQ}{\mathbb{M}_{d \times q}}
\newcommand{\EE}{\mathbb{E}\,}
\newcommand{\conright}{\preceq_{\,cv}}
\newcommand{\conleft}{\succeq{\,cv}}

\newcommand{\mconright}{\preceq_{\,\text{mcv}}}
\newcommand{\mconleft}{\succeq_{\,\text{mcv}}}
\newcommand{\PRR}{\mathcal{P}(\RR)}
\newcommand{\PPRR}{\mathcal{P}_{p}(\RR)}

\newcommand{\Lip}{\text{Lip}}

\newcommand{\EMH}{\mathcal{E}_{m}^{h}}

\vspace{-0.8cm}
\begin{abstract}
This paper is a continuation of our previous paper~\cite{liu2020functional}.
In this paper, we establish the \textit{monotone} convex order (see further \eqref{defmcv1}) between two $\RR$-valued McKean-Vlasov processes $X=(X_t)_{t\in [0, T]}$ and $Y=(Y_t)_{t\in [0, T]}$ defined on a filtered probability space $(\Omega, \mathcal{F}, (\mathcal{F}_{t})_{t\geq0}, \mathbb{P})$ by 
\[\begin{cases}
dX_{t}=b(t, X_{t}, \mu_{t})dt+\sigma(t, X_{t}, \mu_{t})dB_{t}, \quad X_{0}\in  L^{p}(\mathbb{P})\; \text{with}\; p\geq 2,\\
dY_{t}\,=\beta(t, Y_{t}, \,\nu_{t})dt+\theta(t, \,Y_{t}, \,\nu_{t})\,dB_{t}, \,\quad Y_{0}\in L^{p}(\mathbb{P}), \\
\text{where} \:\:\forall\, t\in [0, T],\: \mu_{t}=\mathbb{P}\circ X_{t}^{-1}, \:\nu_{t}=\mathbb{P}\circ Y_{t}^{-1}. 
\end{cases}\]
If we make the convexity and monotony assumption (only) on $b$ and $|\sigma|$ and if $b\leq \beta$ and $|\sigma|\leq |\theta|$, then the monotone convex order for the initial random variable $X_0\mconright Y_0$ can be propagated to the whole path of processes $X$ and $Y$. That is, if we consider a non-decreasing convex functional $F$ defined on the path space with polynomial growth, we have $\EE F(X)\leq \EE F(Y)$; for a non-decreasing convex functional $G$ defined on the product space involving the path space and its marginal distribution space, we have $\EE G(X, (\mu_{t})_{t\in [0, T]})\leq \EE G(Y, (\nu_{t})_{t\in [0, T]})$ under appropriate conditions. The symmetric setting is also valid, that is, if $Y_0\mconright X_0$ and $|\theta|\leq |\sigma|$, then $\EE F(Y)\leq  \EE F(X)$ and $\EE G(Y, (\nu_{t})_{t\in [0, T]})\leq \EE G(X, (\mu_{t})_{t\in [0, T]})$.  The proof is based on several forward and backward dynamic programming principle and the convergence of the \textit{truncated} Euler scheme of the McKean-Vlasov equation. 
\end{abstract}


\section{Introduction} 

Let $U, V: (\Omega, \mathcal{F}, \mathbb{P})\rightarrow \big(\RR, Bor(\RR)\big)$ be two integrable random variables. We call $U$ is dominated by $V$ for the \textit{monotone convex order} - denoted by $U\mconright V$ - if for every non-decreasing convex function $\varphi: \RR\rightarrow\RR$,
\begin{equation}\label{defmcv1}
\EE \varphi(U)\leq \varphi(V).
\end{equation}
Let $\mathcal{P}(\RR)$ denote the set of all probability distributions on $(\RR, Bor(\RR)\big)$ and for every $p\in[1, +\infty)$, let 
\[\mathcal{P}_{p}(\RR)\coloneqq \Big\{\mu\in\mathcal{P}(\RR) \text{ s.t. }\int_{\RR}\left|\xi\right|^{p}\mu(d\xi)<+\infty\Big\}.\]
We can naturally generalize the definition of the monotone convex order on $\mathcal{P}_{1}(\RR)$: for any $\mu, \nu\in\mathcal{P}_{1}(\RR)$, we denote $\mu\mconright\nu$ if for every non-decreasing convex function $\varphi$, 
\begin{equation}\label{defmcv2}
\int_{\RR}\varphi(\xi)\mu(d\xi)\leq \int_{\RR}\varphi(\xi)\nu(d\xi).
\end{equation}
For any random variable $X$, if we denote its probability distribution by $\mathbb{P}_{X}=\mathbb{P}\circ X^{-1}$, it is easy to see that $U\mconright V$ implies $\mathbb{P}_{U}\mconright \mathbb{P}_{V}$ and vice versa. As for what is established in~\cite[Lemma A.1]{alfonsi:hal-01589581} for  (regular) convex order, \textcolor{black}{the following lemma, whose proof  is postponed in Appendix A, shows that monotone convex order can be characterized  by establishing~(\ref{defmcv1}) and~(\ref{defmcv2}) for  a smaller class of monotonic convex functions}. 
\begin{lem}\label{onlylineargrowth} For every $\mu$, $\nu\!\in {\cal P}_1(\RR)$, we have $\mu \mconright \nu $ if and only if, for every convex and non-decreasing function $f:\RR\to \RR$ with linear growth, $\int_{\RR} f d\mu \le \int_{\RR} f d\nu$.
\end{lem}
 
\textcolor{black}{Convex order of distribution has received much attention for a long time theoretically motivated by the study \textcolor{black}{of} martingales since their marginal distributions are ordered  in a convex order and, for more applied matters, by risk management, actuarial sciences and more recently finance. The equivalence between $U\mconright V$ and the existence of a one period martingale  $(M_i)_{i=1,2}$ such that $M_1\sim U$ and $M_2\sim V$ is due to Kellerer  in~\cite{KEL}, see also \cite{Strassen} for a version introducing a martingale transition kernel. These results, whose proofs  are not constructive, were brought back to light  at the end of the 1990's in order to analyze the sensitivity of possibly path-dependent derivative products to volatility when the dynamics of the underlying asset was a geometric Brownian motion and more generally diffusion process when dealing with local volatility models, see~\cite{NEKJS1998} for the robustness of Black-Scholes model, \cite{CarrEW2008, BergenthumR2008} for a first approach to path dependent options among others. Some  of these works \textcolor{black}{were} focused on the exhibition of explicit martingales, see~\cite{Yor} for Asian options and then~\cite{HirschPRY2011} for a  systematic exploration of such martingale representation of convex order. Another direction, still motivated by financial application, was to establish (path-dependent) convex order  for discretization schemes (the Euler scheme in  practice) combined with appropriate (functional)  limit  theorems toward the target diffusion. Such a strategy, systematically developed e.g. in~\cite{pages2016convex} avoids to introduce arbitrage while pricing and hedging derivative products written on convex payoffs. This last approach was extended to McKean-Vlasov equation in~\cite{liu2020functional}.}

\textcolor{black}{As mentioned above, convex order is closely connected to martingality and is subsequently associated to martingale diffusions or to diffusions sharing the same affine drift which can be reduced to martingale up to  an appropriate scaling. When dealing with more general (scalar) diffusions with drifts, Hajek pointed out and proved in~\cite{Hajek1985} that under an additional convexity assumptions on this drift, their marginals can be ordered in monotone convex order as defined above.}

\textcolor{black}{This paper aims at extending this seminal  old result by proving  functional monotone convex order for McKean-Vlasov diffusions whose  drift $b(t,x,\mu)$   are convex in  $x$ (space) and non-decreasing in $\mu$ (distribution). It can be seen as}  the continuation of our paper~\cite{liu2020functional} in which we discuss the (regular) functional convex order for two McKean-Vlasov processes having \textcolor{black}{the same}  affine drift. In this paper, within the framework of  the \textit{monotone} convex order, we can establish such convex order results for two one-dimensional McKean-Vlasov processes with general drift coefficient. 
Let $(\Omega, \mathcal{F}, (\mathcal{F})_{t\geq 0}, \mathbb{P})$ be a filtered probability space satisfying the usual conditions.
Let $X=(X_t)_{t\in[0, T]}$ and $Y=(Y_t)_{t\in[0, T]}$ be two McKean-Vlasov processes, respective solutions to 
\begin{align}
\label{defx}
&dX_{t}=b(t, X_{t}, \mu_{t})dt+\sigma(t, X_{t}, \mu_{t})dB_{t}, \quad X_{0}\in L^{p}(\mathbb{P}),\\
\label{defy}
&dY_{t}\,=\beta(t, Y_{t}, \,\nu_{t})dt+\theta(t, \,Y_{t}, \,\nu_{t})\,dB_{t},  \quad Y_{0}\in L^{p}(\mathbb{P}),
\end{align}
where $p\geq 2$, for every $t\in[0, T]$, $\mu_t=\mathbb{P}\circ X_{t}^{-1}$ and $\nu_t=\mathbb{P}\circ Y_{t}^{-1}$, the coefficients $b, \,\beta,\, \sigma, \,\theta$ are $\RR$-valued functions defined on $\big([0, T], \RR, \mathcal{P}_{1}(\RR)\big)$  and $(B_{t})_{t\in[0, T]}$ denotes an $(\mathcal{F}_{t})$-standard Brownian motion valued in $\mathbb{R}$, independent to $(X_0, \,Y_0)$.

We introduce the definition of the Wasserstein distance $\mathcal{W}_{p}$ on $\PPRR$ as follows:  for every $\mu, \nu\in\PPRR$, $p\geq 1$, 
\vspace{-0.4cm}
\begin{flalign}\label{defwas}
\mathcal{W}_{p}(\mu,\nu)\coloneqq&\Big{(}\inf_{\pi\in\Pi(\mu,\nu)}\int_{\mathbb{R}^2}d(x,y)^{p}\pi(dx,dy)\Big{)}^{\frac{1}{p}}&\nonumber\\
=&\inf\Big{\{}\big{[}\mathbb{E}\,\left|X-Y\right|^{p}\big{]}^{\frac{1}{p}},\,  X,Y:(\Omega,\mathcal{A},\mathbb{P})\rightarrow( \mathbb{R} ,Bor( \mathbb{R}))  \;\text{with} \;\mathbb{P}_{X}=\mu,\; \mathbb{P}_{Y}=\nu\,\Big{\}},&
\end{flalign}
where in the first ligne of~(\ref{defwas}), $\Pi(\mu,\nu)$ denotes the set of all probability measures on $\big( \mathbb{R}^2, Bor( \mathbb{R}^{2})\big)$ with marginals $\mu$ and $\nu$. Let $\delta_{0}$ denote the Dirac mass on 0. 

Throughout this paper, we make the following assumptions:

\noindent\textbf{Assumption (I)}
There exists {\color{black}$p\in[2, \infty)$} such that $\vertii{X_0}_p\vee\vertii{Y_0}_p<+\infty$. The functions $b, \beta, \sigma, \theta$ are $\gamma$-H\"older continuous in $t$ and Lipschitz continuous in $x$ and in $\mu$ in the following sense \footnote{For convenience, we explain these assumptions only for $b$ but these inequalities are assumed for all the four coefficient functions $b, \beta, \sigma, \theta$.}:  for every $s, t\in [0, T]$ with $s\leq t$, 
there exists $C>0$ such that 
\begin{align}\label{assumpholder}
\forall \, x\in\mathbb{R}, \,\forall\mu\in\PPRR, \qquad\left|b(t, x, \mu)-b(s, x, \mu)\right|\leq C\big(1+\left|x\right|+\mathcal{W}_{p}(\mu, \delta_{0})\big)(t-s)^{\gamma}, &
\end{align}
and for every $t\in[0, T]$, there exists $L>0$ such that 
\begin{align}\label{assumplip}
\forall \, x,y \in\mathbb{R},\,\forall \, \mu, \nu\in\mathcal{P}_{p}(\mathbb{R}),\qquad\left|b(t, x, \mu) - b(t, y, \nu)\right|\leq L\big[\left|x-y\right|+\mathcal{W}_{p}(\mu, \nu)\big].
\end{align}

\noindent  \textbf{Assumption (II)}
\vspace{-0.5cm}
\begin{enumerate}[(1)]
\item For every fixed $t\in[0, T]$ and \textcolor{black}{every $\mu\in\mathcal{P}_{1}(\RR)$,  $b(t, \cdot, \mu)$ and $\left|\sigma(t, \cdot, \mu)\right|$ are convex functions (in $x$) on $\RR$}.
\item For every fixed $(t, x)\in[0, T]\times \RR$, the functions $b(t, x, \cdot)$ and $\sigma(t, x, \cdot)$ are non-decreasing in $\mu$ with respect to the monotone convex order in the sense that 
\[\mu, \nu\in\mathcal{P}_{1}(\RR), \:\mu\mconright\nu \quad\Longrightarrow\quad b(t, x, \mu)\leq b(t,x, \nu),\quad {\color{black}\left|\sigma(t, x, \mu)\right|\leq \left|\sigma(t,x, \nu)\right|}.\]
\item For every $(t, x,\mu)\in [0, T]\times \RR\times \mathcal{P}_{1}(\RR)$, we have 
\[b(t, x, \mu)\leq \beta(t, x, \mu)\quad \text{and}\quad{\color{black}\left|\sigma(t, x, \mu)\right|\leq \left|\theta(t, x, \mu)\right|}.\]
\item $X_0\mconright Y_0$.
\end{enumerate}


{\color{black}
\begin{rem}   Note  that in \textcolor{black}{Assumptions (II), the conditions are made on $\left|\sigma\right|$ and $\left|\theta\right|$ rather than on $\sigma$ and $\theta$ since, e.g. for $\sigma$,}  the   process $(X_{t}')_{t\in [0, T]}$ defined by
\[\begin{cases}
\displaystyle \;X'_{t}= X_{0}+\int_{0}^{t}b(s, X'_{s}, \mu'_{s})ds+\int_{0}^{t}\left|\sigma(s, X'_{s}, \mu'_{s})\right|dB_{s}, \\
 \;\forall\, t\in [0, T],\;\mu'_{t}=\mathbb{P}\circ X_{t}^{'-1}
\end{cases}\]
has the same distribution as $(X_{t})_{t\in [0, T]}$ defined by (\ref{defx}). 
\end{rem}

\noindent {\bf Example: the Vlasov case.} As for the Vlasov case, that is, if there exist four functions $\boldsymbol{b}, \boldsymbol{\beta}, \boldsymbol{\sigma}, \boldsymbol{\theta}: [0, T]\times \RR \times \RR\rightarrow \RR$ such that for every $(t, x, \mu)\in [0, T]\times \RR\times \mathcal{P}_{1}(\RR)$
\begin{align}
&b(t, x, \mu)=\int_{\RR}\boldsymbol{b} (t, x, y)\mu(dy),  &\sigma(t, x, \mu)=\int_{\RR}\boldsymbol{\sigma} (t, x, y)\mu(dy),\nonumber\\
&\beta(t, x, \mu)=\int_{\RR}\boldsymbol{\beta} (t, x, y)\mu(dy), & \theta(t, x, \mu)=\int_{\RR}\boldsymbol{\theta} (t, x, y)\mu(dy),\nonumber
\end{align}
\textcolor{black}{we have the following natural sufficient conditions for Assumption (I)}
\vspace{-0.3cm}
\begin{enumerate}[$(1)_{v}$]
\item  $\boldsymbol{b}, \boldsymbol{\beta}, \boldsymbol{\sigma}, \boldsymbol{\theta}$ are $\gamma$-H\"older continuous in $t$ and Lipschitz continuous in $x$ and in $y$. 
\end{enumerate}
 \vspace{-0.3cm}
\noindent and for Assumption (II)-(1), (2), (3) 
 \vspace{-0.3cm}
\begin{enumerate}[$(1)_{v}$]
\item  $\boldsymbol{b}$, $\boldsymbol{\sigma}$ are convex in $x$, 
\item  $\boldsymbol{b}$, $\boldsymbol{\sigma}$ are convex and non-decreasing in $y$, 
\item $\boldsymbol{b}\leq \boldsymbol{\beta}$ and 
$|\boldsymbol{\sigma}|\leq \boldsymbol{\theta}$.
\end{enumerate}

Let $\big(\mathcal{C}([0, T], \RR),\; \vertii{\cdot}_{\sup}\big)$ denote the set of all continuous functions defined on $[0, T]$ and valued in $\RR$ equipped with the supremum norm $\vertii{\alpha}_{\sup}\coloneqq \sup_{t\in[0, T]}\left|\alpha_{t}\right|$.
In the literature on the McKean-Vlasov equation (see e.g.~\cite[Section 5.1]{liu:tel-02396797},~\cite[Theorem 3.3]{lacker2018mean}), Assumption (I) is a classical assumption for  existence and strong uniqueness of the solutions $X, Y: (\Omega, \mathcal{F},\mathbb{P})\rightarrow \big(\mathcal{C}([0, T], \RR), \vertii{\cdot}_{\sup}\big)$ of~(\ref{defx}) and~(\ref{defy}).  Moreover, let 
\begin{align}\label{margspace}
\mathcal{C}\big([0, T], \mathcal{P}_{p}(\RR)\big)\coloneqq \Big\{(\mu_{t})_{t\in [0, T]} &\text{ such that the mapping $t\mapsto \mu_{t}$} \nonumber\\& \quad \text{  is continuous w.r.t. the Wasserstein distance $\mathcal{W}_{p}$}\Big\}
\end{align}
denote the marginal distribution space, equipped with the distance 
\begin{equation}\label{disdc}
d_{p}\Big((\mu_{t})_{t\in[0, T]}, \,(\nu_{t})_{t\in [0, T]}\Big)\coloneqq \sup_{t\in[0, T]}\mathcal{W}_{p}(\mu_{t}, \nu_{t}).
\end{equation}
It follows from~\cite[Lemma 3.2]{liu2020functional} that the marginal distributions $(\mu_{t})_{t\in [0, T]}$ and $(\nu_{t})_{t\in [0, T]}$  of the solution processes $X$ and $Y$ of~(\ref{defx}) and~(\ref{defy}) lie in $\big(\mathcal{C}\big([0, T], \mathcal{P}_{p}(\RR)\big) , d_{p}\big)$. 
Moreover, we define a pointwise partial order $\preceq$ on $\mathcal{C}([0, T], \RR)$ as follows, 
\begin{equation}\label{partialorder1} 
\forall \,\alpha=(\alpha_t)_{t\in[0, T]},\; \beta=(\beta_{t})_{t\in [0, T]}\in \mathcal{C}\big([0, T], \RR\big), \: \text{we denote by }  \alpha\preceq\beta \,\text{ if }\, \forall\, t\in[0, T],\, \alpha_t\leq \beta_t. 
\end{equation}

The main result of this paper is the following theorem.
\begin{thm}\label{main}
Let $X\coloneqq (X_t)_{t\in[0, T]}, \:Y\coloneqq (Y_t)_{t\in[0, T]}$ respectively denote the unique solution of the McKean-Vlasov equations~(\ref{defx}),~(\ref{defy}) and let $(\mu_t)_{t\in [0, T]}$, $(\nu_t)_{t\in [0, T]}$ respectively denote the marginale distributions of $(X_t)_{t\in[0, T]}$, $(Y_t)_{t\in[0, T]}$. Under Assumption (I) and (II),
\begin{enumerate}[$(a)$]
\item For any functional $F: \big(\mathcal{C}([0, T], \RR), \vertii{\,\cdot\,}_{\sup}\big)\rightarrow \RR$ satisfying the following conditions 

\begin{enumerate}[$(i)$]
\item $F$ is convex and has an $r$-polynomial growth, $1\leq r < p$ in the sense that there exists a constant $C>0$ such that 
\[\forall \alpha\in\mathcal{C}([0, T], \RR), \quad \left|F(\alpha)\right|\leq C\big(1+ \sup_{t\in[0, T]}\left|\alpha_{t}\right|\big),\]
\item $F$ is non-decreasing with respect to the partial order~(\ref{partialorder1}),
\end{enumerate}
one has 
\begin{equation}\label{result1}
\EE F(X)\leq \EE F(Y).
\end{equation}
\item For any functional 
\[G:  \big(\alpha, (\eta_{t})_{t\in[0, T]}\big)\in\mathcal{C}([0, T], \RR)\times\mathcal{C}\big([0, T], \mathcal{P}_{1}(\RR)\big) \longmapsto G\big(\alpha, (\eta_{t})_{t\in[0, T]}\big)\in \RR\] satisfying the following conditions:
\begin{enumerate}[$(i)$]
\item $G$ is convex and non-decreasing with respect to the partial order~(\ref{partialorder1}) in $\alpha$, 
\item $G$ has an $r$-polynomial growth, $1\leq r< p$, in the sense that 
\begin{flalign}\label{rpolygrowth}
&\exists \,C\in\mathbb{R}_{+}\text{ such that }\forall \, \big(\alpha, (\eta_{t})_{t\in[0, T]}\big)\in\mathcal{C}\big([0, T], \RR\big)\times \mathcal{C}\big([0, T], \mathcal{P}_{p}(\RR)\big),& \nonumber\\
&\hspace{3.5cm}G\big(\alpha, (\eta_{t})_{t\in[0, T]}\big)\leq C\big[1+\vertii{\alpha}_{\sup}^{r}+\sup_{t\in[0, T]}\mathcal{W}_{p}^{\,r}(\eta_{t}, \delta_{0})\big],&
\end{flalign}

\vspace{-0.2cm}
\item $G$ is continuous in $(\eta_{t})_{t\in[0, T]}$ with respect to the distance $d_{r},\, 1\leq r<p$ defined in~(\ref{disdc}) and non-decreasing in $(\eta_{t})_{t\in[0, T]}$ with respect to the monotone convex order in the sense that 
\begin{flalign}
&\forall \alpha\in \mathcal{C}\big([0, T], \RR\big),\: \forall \, (\eta_{t})_{t\in[0, T]}, (\tilde{\eta}_{t})_{t\in[0, T]}\in \mathcal{C}\big([0, T], \mathcal{P}_{r}(\RR)\big) \text{ s.t. } \forall \, t \!\in[0, T],\,\eta_{t}\mconright\tilde{\eta}_{t}, &\nonumber\\
&\hspace{3.5cm} \quad \;G\big(\alpha,  (\eta_{t})_{t\in[0, T]}\big)\leq G\big(\alpha,  (\tilde{\eta}_{t})_{t\in[0, T]}\big),&\nonumber
\end{flalign}
\end{enumerate}

\vspace{-0.6cm}
\noindent one has 
\begin{equation}\label{result2}
\EE G\big(X, (\mu_{t})_{t\in[0, T]}\big)\leq \EE G\big(Y, (\nu_{t})_{t\in[0, T]}\big). 
\end{equation}

\end{enumerate}
\end{thm}

Moreover, the symmetric case of Theorem \ref{main} remains true, that is, if we replace Assumption (II) by the following Assumption (II'),

\noindent  \textbf{Assumption (II')} (Symmetric setting)
\vspace{-0.2cm}
\begin{enumerate}
\item[$(1')$]  and $(2')$ are same as Assumption (II) - (1) and (2),
\item[$(3')$] For every $(t, x,\mu)\in [0, T]\times \RR\times \mathcal{P}_{1}(\RR)$, we have 
\[\beta(t, x, \mu) \leq  b(t, x, \mu)\quad \text{and}\quad{\color{black}\left|\theta(t, x, \mu)\right|\leq\left|\sigma(t, x, \mu)\right| },\]
\item[$(4')$] $Y_0\mconright X_0$,
\end{enumerate}
we have the following theorem whose proof is very similar to that of Theorem \ref{main}.

\begin{thm} [Symmetric setting] \label{sym} Under Assumption (I) and (II'), if we consider two functionals  $F: \mathcal{C}([0, T], \RR)\rightarrow \RR$ and $G:  \mathcal{C}([0, T], \RR)\times\mathcal{C}\big([0, T], \mathcal{P}_{1}(\RR)\big) \rightarrow  \RR$ respectively satisfying the conditions in Theorem \ref{main} - $(a)$ and $(b)$, then 
\vspace{-0.1cm}
\[\EE F(Y)\leq \EE F(X) \quad\text{and}\quad  \EE G\big(Y, (\nu_{t})_{t\in[0, T]}\big)\leq \EE G\big(X, (\mu_{t})_{t\in[0, T]}\big).  \]
\end{thm}

This paper is organised as follows. 

\textcolor{black}{We adopt the strategy already used in~\cite{liu2020functional}: we establish our results in two steps, first in a discrete time setting  for the Euler scheme of the McKean-Vlasov equations under consideration   and then transferring results by using the  convergence of this schemes. However, in full generality the Euler scheme does not propagate monotonicity so we are led in Section~\ref{euler}, to introduce the truncated Euler scheme reading, e.g. for the McKean-Vlasov process $(X_t)_{t\in[0,T]}$\footnote{ \textcolor{black}{The truncated Euler scheme of $Y$ is defined likewise with $\beta$ and $\theta$.}},}
\begin{equation}\label{trunc}
\begin{cases}
\widetilde{X}^{M}_{t_{m+1}}=\widetilde{X}^{M}_{t_{m}}+h\cdot b(t_{m}, \widetilde{X}^{M}_{t_{m}}, \widetilde{\mu}^{M}_{t_{m}}\,)+\sqrt{h}\cdot\sigma(t_{m}, \widetilde{X}^{M}_{t_{m}}, \widetilde{\mu}^{M}_{t_{m}}\,)Z^{h}_{m+1},\\
 \text{where } Z^{h}_{m+1}  \text{ is a \textcolor{black}{symmetric} truncated Gaussian random variable}\end{cases}\end{equation}
that we use to establish the propagation of convexity and monotonicity of the processes $X$ and $Y$.
Next, in Section \ref{mcvEuler}, we prove that when the time step $h=\frac{T}{M}$ is small enough, the truncated Euler scheme~(\ref{trunc}) propagates the monotone convex order in a marginal sense, namely, 
\[\forall \;0\leq m\leq M, \quad \widetilde{X}^{M}_{t_{m}}\mconright \widetilde{Y}^{M}_{t_{m}}\]
by a forward dynamic programming principle and in a functional sense, namely, 
 \begin{align}
 &\forall \,F: \RR^{M+1}\rightarrow \RR \;\text{non-decreasing,  convex and having $r$-polynomial growth, $1\leq r\leq p$ }\nonumber\\
 &\hspace{4cm} \EE\, F(\widetilde{X}_{0},\ldots, \widetilde{X}_{M})\leq \EE\, F(\widetilde{Y}_{0},\ldots, \widetilde{Y}_{M})\nonumber \end{align}
by a backward dynamic programming principle.  Finally, in Section 3, on  prove the convergence of the truncated Euler scheme~(\ref{trunc}) and we prove Theorem~\ref{main}.

\subsection{Application: convex partitioning and convex bounding}\label{appli}

Theorem~\ref{main} and \ref{sym} show that we can \textit{upper and lower bound}, with respect to the monotone functional convex order, a McKean-Vlasov process by two McKean-Vlasov processes satisfying Assumption (I), (II), or we can \textit{separate}, with respect to the monotone functional convex order, two McKean-Vlasov processes by a McKean-Vlasov process satisfying Assumption (I), (II). Consider the following McKean-Vlasov equations 
\begin{align}
& dX^{\boldsymbol{1}}_{t}=b_{1}(t, X^{\boldsymbol{1}}_{t}, \mu^{\boldsymbol{1}}_{t})dt+\sigma_{1}(t, X^{\boldsymbol{1}}_{t}, \mu^{\boldsymbol{1}}_{t})dB_{t},\quad  dX^{\boldsymbol{2}}_{t}=b_{2}(t, X^{\boldsymbol{2}}_{t}, \mu^{\boldsymbol{2}}_{t})dt+\sigma_{2}(t, X^{\boldsymbol{2}}_{t}, \mu^{\boldsymbol{2}}_{t})dB_{t},\nonumber\\
& dY^{\boldsymbol{1}}_{t}=\beta_{1}(t, Y^{\boldsymbol{1}}_{t},\, \nu^{\boldsymbol{1}}_{t})dt+\theta_{1}(t, Y^{\boldsymbol{1}}_{t}, \,\nu^{\boldsymbol{1}}_{t})dB_{t},\quad  dY^{\boldsymbol{2}}_{t}=\beta_{2}(t, Y^{\boldsymbol{2}}_{t},\, \nu^{\boldsymbol{2}}_{t})dt+\theta_{2}(t, Y^{\boldsymbol{2}}_{t},\, \nu^{\boldsymbol{2}}_{t})dB_{t}\nonumber
\end{align}
with the coefficients  $b_1,\; \sigma_{1}, \;b_2,\;\sigma_{2}$ satisfying Assumption (II)-(1),(2) and consider two functionals $F$ and $G$ satisfying conditions in Theorem~\ref{main} -(a), (b) respectively. We have the following inequalities :

\vspace{-0.3cm}
 \begin{itemize}
\item[$-$] (\textit{Monotone convex bounding inequality}) If $X^{\boldsymbol{1}}_{0}\preceq_{\,mcv}Y^{\boldsymbol{1}}_{0}\preceq_{\,mcv}X^{\boldsymbol{2}}_{0}$, $b_1\leq \beta_1\leq b_2$ and $|\sigma_1|\leq |\theta_1|\leq |\sigma_2|$, then 
\begin{equation}\label{bounding}
\begin{cases}
\EE F(X^{\boldsymbol{1}}) \leq\EE F(Y^{\boldsymbol{1}})\leq \EE F(X^{\boldsymbol{2}}), \\
 \EE G\big(X^{\boldsymbol{1}}, (\mu^{\boldsymbol{1}}_{t})_{t\in[0, T]}\big)\leq\EE G\big(Y^{\boldsymbol{1}}, (\nu^{\boldsymbol{1}}_{t})_{t\in[0, T]}\big)\leq \EE G\big(X^{\boldsymbol{2}}, (\mu_{t}^{\boldsymbol{2}})_{t\in[0, T]}\big);\!\!\!\!\!\!\end{cases}\end{equation}
\item[$-$] (\textit{Monotone convex partitioning inequality}) If $Y^{\boldsymbol{1}}_{0}\preceq_{\,mcv}X^{\boldsymbol{2}}_{0}\preceq_{\,mcv}Y^{\boldsymbol{2}}_{0}$, $\beta_1\leq b_2\leq \beta_2$ and $|\theta_1|\leq |\sigma_2|\leq |\theta_2|$, then 
\begin{equation}\label{partitioning}
\begin{cases}
\EE F(Y^{\boldsymbol{1}})\leq \EE F(X^{\boldsymbol{2}})\leq\EE F(Y^{\boldsymbol{2}}),\\
\EE G\big(Y^{\boldsymbol{1}}, (\nu^{1}_{t})_{t\in[0, T]}\big)\!\leq\! \EE G\big(X^{\boldsymbol{2}}, (\mu^{2}_{t})_{t\in[0, T]}\big)\!\leq\! \EE G\big(Y^{\boldsymbol{2}}, (\nu^{2}_{t})_{t\in[0, T]}\big).\!\!\!\!\!\!\end{cases}
\end{equation}
\end{itemize}
\noindent If we can find appropriate functions $b_1, b_2, \sigma_1, \sigma_2$ which are convex in $x$, do not depend on $\mu$ and satisfy Assumption (II), the results in~(\ref{bounding}) and~(\ref{partitioning}) make a connection between the McKean-Vlasov equation and the regular Brownian diffusion, the latter is much easier to simulate. To be more precise, for the monotone convex bounding (\ref{bounding}), consider the following two examples 
\begin{align}
&\text{Example 1 :}\:\;\begin{cases}dY_{t}=0.05\cdot Y_t\big[\EE\sin^2(Y_t)+2\,\big]dt+ Y_{t}\,dB_{t},\qquad Y_0=1,\\
dX_{t}^{\text{down}}=0.05\cdot X_{t}^{\text{down}}dt+ X_t^{\text{down}} \, dB_{t}, \qquad X_{0}^{\text{down}}=1,\\ dX_{t}^{\text{up}}=0.15\cdot X_{t}^{\text{up}}dt+ X_t^{\text{up}}\, dB_{t}, \qquad X_{0}^{\text{up}}=1,
\end{cases}\nonumber\\
\nonumber\\
&\text{Example 2 :}\nonumber\\
&\begin{cases}dY_{t}'=0.05\cdot \Big[ \log\big(\cosh(Y_{t})\cdot [1.5+\EE \sin(Y_{t})]\big)+2\Big]dt\\
\hspace{3cm}+0.3\cdot\Big[ \log\big(\cosh(Y_{t})\cdot [1.5+\EE \cos(Y_{t})]\big)+2\Big]dB_{t},\quad Y_0=2,\\
dX_{t}^{' \text{down}}\!=0.05\!\cdot\! \Big[ \log\big(\cosh(X_{t}^{' \text{down}})\big)+   1.306\Big]dt \!+\! 0.3\!\cdot\!  \Big[ \log\big(\!\cosh(X_{t}^{' \text{down}})\big)+ \!  1.306\Big]dB_{t},\: X_{0}^{' \text{down}}=2,\\
dX_{t}^{' \text{up}}=0.05\cdot \Big[ \log\big(\cosh(X_{t}^{' \text{up}})\big)+   2.917\Big]dt + 0.3\cdot  \Big[ \log\big(\cosh(X_{t}^{' \text{up}})\big)+  2.917\Big]dB_{t}, \quad X_{t}^{' \text{up}}=2,
\end{cases}\nonumber
\end{align}
and convex functionals $F_{t}(\alpha)=\max(\alpha_t,0)^{2}, \:t\in[0, 1]$. Theorem \ref{main} and \ref{sym} directly imply 
\begin{equation}\label{example1}
\forall \; t\in[0, 1], \quad\EE F_t(X^{\text{down}})\leq \EE F_t(Y)\leq \EE F_t(X^{ \text{up}}), \quad \EE F_t(X^{' \text{down}})\leq \EE F_t(Y')\leq \EE F_t(X^{' \text{up}}).\end{equation}
We show in the following figures the simulation of \eqref{example1} for every $ t\in[0, 1]$. Note that in the following simulation, $\EE F_t(X^{\text{down}})$ and $\EE F_t(X^{\text{up}})$ have explicit formula
\[\EE F_t(X^{\text{down}})= e^{1.1t},\quad \EE F_t(X^{\text{up}})=e^{1.3t}\] 
since $X^{\text{down}}$ and $X^{\text{up}}$ are geometric Brownian motions. Moreover, the two processes $t\mapsto\EE F_t(Y)$ and $t\mapsto\EE F_t(Y')$ are simulated by using the particle method (see e.g. \cite{bossy1997stochastic} and \cite[Section 7.1]{liu:tel-02396797}), $t\mapsto\EE F_t(X^{' \text{down}})$ and $t\mapsto\EE F_t(X^{' \text{up}})$ are simulated by using the Monte Carlo method.

\smallskip
\definecolor{darkgreen}{rgb}{0.0, 0.2, 0.13}
\definecolor{cadmiumgreen}{rgb}{0.0, 0.42, 0.3}
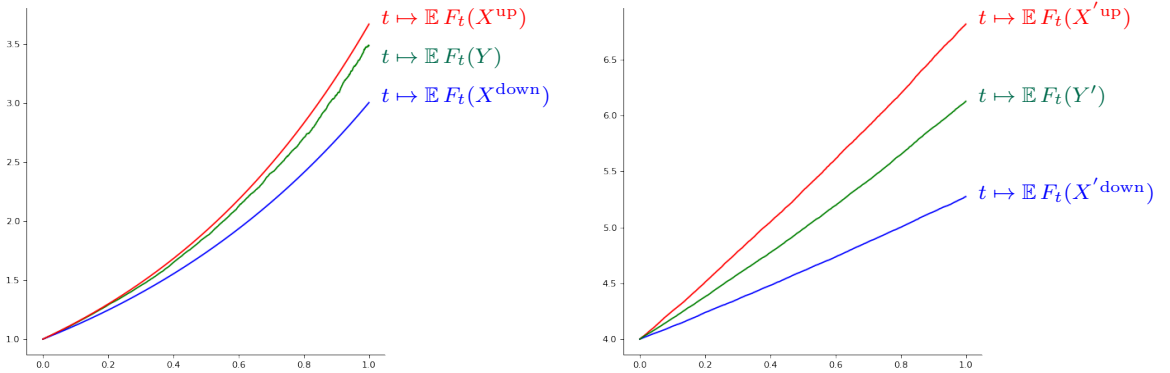
\begin{figure}[h]
\begin{minipage}[c]{0.5\textwidth}
\begin{overpic}[width=0.7\textwidth]{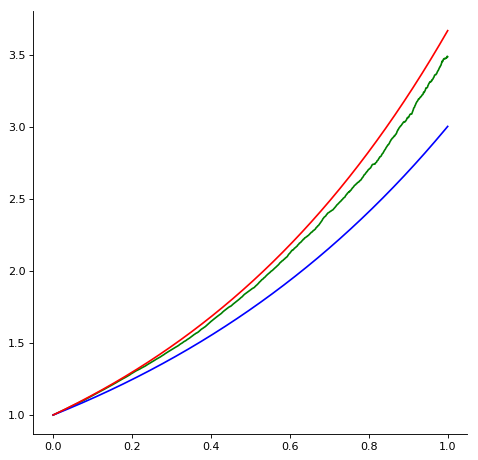}
\put(97,90){\footnotesize\bfseries\color{red}{$t\mapsto \EE F_t(X^{\mathrm{up}}) $}}
\put(97,80){\footnotesize\bfseries\color{cadmiumgreen}{$t\mapsto \EE F_t(Y) $}}
\put(97,70){\footnotesize\bfseries\color{blue}{$t\mapsto \EE F_t(X^{\mathrm{down}}) $}}
\end{overpic}
\end{minipage}
\hspace{0.3cm}
\begin{minipage}[c]{0.5\textwidth}
\begin{overpic}[width=0.7\textwidth]{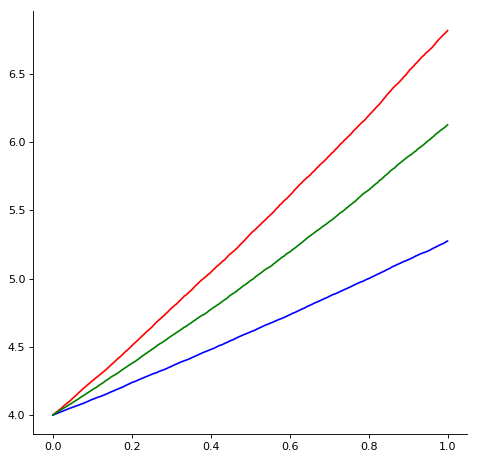}
\put(97,90){\footnotesize\bfseries\color{red}{$t\mapsto \EE F_t(X^{'\mathrm{up}}) $}}
\put(97,70){\footnotesize\bfseries\color{cadmiumgreen}{$t\mapsto \EE F_t(Y') $}}
\put(97,45){\footnotesize\bfseries\color{blue}{$t\mapsto \EE F_t(X^{'\mathrm{down}}) $}}
\end{overpic}
\end{minipage}
\caption{Simulation of \eqref{example1} for Example 1 (\textit{left}) and Example 2 (\textit{right}).}
\end{figure}

\subsection{Regular and truncated Euler scheme}\label{euler}

In our previous paper~\cite{liu2020functional}, we proved that the strong $L^p$-convergence (including the convergence rate) of the following regular Euler scheme of~(\ref{defx}) under Assumption (I):
\begin{equation}\label{Eulerclass}
\begin{cases}
\bar{X}^{M}_{t_{m+1}}=\bar{X}^{M}_{t_{m}}+h\cdot b(t_{m}, \bar{X}^{M}_{t_{m}}, \bar{\mu}^{M}_{t_{m}})+\sqrt{h}\cdot\sigma(t_{m}, \bar{X}^{M}_{t_{m}}, \bar{\mu}^{M}_{t_{m}})Z_{m+1},  \\
\bar{X}_{0}=X_{0}; \;\forall\; m=0,\ldots, M-1,\, \bar{\mu}_{t_m}\coloneqq \mathbb{P}\circ  \bar{X}_{t_m}^{-1} \text{ and } Z_{m+1}\coloneqq\frac{1}{\sqrt{h}}(B_{t_{m+1}}-B_{t_{m}}),
\end{cases}
\end{equation}
where $M\in\mathbb{N}^{*}$ is the number of time discretization, $h\coloneqq\frac{T}{M}$ is the time step, $t_{m}=mh,\;m=0,\ldots, M$. 
It is obvious to see that $Z_{m}, m=1, \dots, M, $ are i.i.d random variables having distribution $\mathcal{N}(0, 1)$. When there is no ambiguity, we drop $M$ and $t$ in the superscript and subscript and write $\bar{X}_{m}$ or $\bar{X}_{m}^{M}$ instead of $\bar{X}^{M}_{t_{m}}$. 

As in~\cite[Lemma 2.3 and 2.4]{liu2020functional}, our proof of the monotone convex order results deeply relies on the propagation of convexity and monotonicity by the transition of a discretization scheme.  Unfortunately, {\color{black} the regular Euler scheme~(\ref{Eulerclass}) propagates the monotonicity only if  $|\sigma|$ is non-decreasing in $x$.}
\begin{prop}\label{ifnondecresing}
Assume conditions in Assumption (I) and (II)-(1),(2) are in force. {\color{black}Let $h\in (0, \frac{1}{[b]_{\mathrm{Lip}_{x}}})$ where $[b]_{\mathrm{Lip}_{x}}$ denotes the Lipschitz constant of the coefficient $b$ in $x$. Let $f:\RR \rightarrow\RR$ be a non-decreasing convex function with sub-exponential growth.  If  
 the function $\left|\sigma\right|$ is non-decreasing in $x$ for every $(t, \mu)\in[0, T]\times \mathcal{P}_{1}(\RR)$, then 
 for a fixed $\mu\in\mathcal{P}_{1}(\RR)$, the functions   
 \[x\longmapsto \EE\Big[ f\big(x+h\cdot b(t_{m}, x, \mu)+\sqrt{h}\cdot \sigma(t_{m}, x, \mu)Z_{m+1}\big)\Big],\quad m=0, \dots , M-1,\]
 are also non-decreasing convex functions. }
\end{prop}
\noindent {\color{black}The proof of the above proposition and a counterexample when $\sigma$ is positive and decreasing in $x$ are postponed in Appendix B. To drop the monotony assumption on $\left|\sigma\right|$,  we use in this paper a \textit{truncated} Euler scheme, instead of~(\ref{Eulerclass}), to establish the monotone convex order result (Theorem~\ref{main}).}
We first define the truncation function $T^{\,h}:\RR\rightarrow\big[-\frac{1}{2\sqrt{h}[\sigma]_{\text{Lip}_{x}}}, \frac{1}{2\sqrt{h}[\sigma]_{\text{Lip}_{x}}}\big]$ by 
\begin{equation}\label{truncated}
T^{\,h}(z)\coloneqq z\mathbbm{1}_{\big\{\left|z\right|\leq \frac{1}{2\sqrt{h}[\sigma]_{\text{Lip}_{x}}}\big\}},
\end{equation} 
\noindent where $[\sigma]_{\text{Lip}_{x}}$ is the Lipschitz constant of $\sigma$  in $x$. 
The \textit{truncated} Euler \textcolor{black}{schemes of $X$ and $Y$ with step $h$ are  defined respectively} by 
\begin{align}\label{EulertruncatedX}
&\widetilde{X}^{M}_{t_{m+1}}=\widetilde{X}^{M}_{t_{m}}+h\cdot b(t_{m}, \widetilde{X}^{M}_{t_{m}}, \widetilde{\mu}^{M}_{t_{m}}\,)+\sqrt{h}\cdot\sigma(t_{m}, \widetilde{X}^{M}_{t_{m}}, \widetilde{\mu}^{M}_{t_{m}}\,)Z^{h}_{m+1}, \quad \widetilde{X}^{M}_{0}=X_{0},\\
\label{EulertruncatedY}
&\widetilde{Y}^{M}_{t_{m+1}}\,=\,\widetilde{Y}^{M}_{t_{m}}\,+h\cdot \beta(t_{m}, \widetilde{Y}^{M}_{t_{m}},\, \widetilde{\nu}^{M}_{\,t_{m}}\,)+\sqrt{h}\cdot\theta(t_{m}, \widetilde{Y}^{M}_{t_{m}}, \,\widetilde{\nu}^{M}_{\,t_{m}}\,)Z^{h}_{m+1}, \,\quad \widetilde{Y}^{M}_{0}\,=Y_{0}\,,
\end{align}
where for every $m=1,\ldots, M$, $\widetilde{\mu}_{t_m}$, $\widetilde{\nu}_{t_m}$ denote the respective probability distributions of $\widetilde{X}_{t_m}$, $\widetilde{Y}_{t_m}$ and
\begin{equation}\label{Ztruncated}
Z_m^h=T^{\,h}(Z_{m})\quad \text{with (the same) } Z_{m} \text{ defined in~(\ref{Eulerclass}).} 
\end{equation} 
When there is no ambiguity, we drop the $M$ and $t$ in the superscript and subscript and write $\widetilde{X}_{m}$, $\widetilde{Y}_{m}$ instead of $\widetilde{X}^{M}_{t_{m}}$, $\widetilde{Y}^{M}_{t_{m}}$.



\section{Monotone convex order for the truncated Euler scheme}\label{mcvEuler}
For any $m\in\mathbb{N}$, we define a partial order $\preceq$ on $\RR^{m+1}$ by 
\begin{equation}\label{partialorder}
(x_{0},\ldots, x_{m})\preceq (y_{0},\ldots, y_{m}), \quad\text{ if }\;\forall \; 0\leq i \leq m, \;x_i\leq y_i.
\end{equation} 
Moreover, we call a function $F: \RR^{m+1}\rightarrow \RR$ has an \textit{$r$-polynomial growth} if there exists a constant $C>0$ such that 
\begin{equation}\label{rpolygrowth} 
\forall x=(x_0,\ldots, x_m)\in \RR^{m+1},\quad\left|F(x)\right|\leq C\big(1+\sup_{0\leq i\leq m}\left|x_{i}\right|^{r}\big).
\end{equation}
The main result of this section is the monotone convex order for the truncated Euler scheme, described in the following proposition. 
\begin{prop}\label{funmonoconvexEuler}
Let $h\in (0, \frac{1}{2[b]_{\mathrm{Lip}_{x}}})$.  Let $(\widetilde{X}_{m})_{0\leq m\leq M}$ and $(\widetilde{Y}_{m})_{0\leq m\leq M}$ denote the truncated Euler schemes~(\ref{EulertruncatedX}) and~(\ref{EulertruncatedY}). Let $F: \RR^{M+1}\rightarrow\RR$ be a convex function having an $r$-polynomial growth, $1\leq r\leq p$, non-decreasing with respect to the partial order defined in~(\ref{partialorder}). 
\vspace{-0.2cm}
\begin{enumerate}[$(a)$]
\item Under Assumption (I) and (II), we have $ \EE\, F(\widetilde{X}_{0},\ldots, \widetilde{X}_{M})\leq \EE\, F(\widetilde{Y}_{0},\ldots, \widetilde{Y}_{M})$.
\item {\color{black}(Symmetric setting) Under Assumption (I) and (II'),} we have \[\EE\, F(\widetilde{Y}_{0},\ldots, \widetilde{Y}_{M})\leq \EE\, F(\widetilde{X}_{0},\ldots, \widetilde{X}_{M}).\]
\end{enumerate}
\end{prop}

Remark that, as $F$ has an $r$-polynomial growth, $1\leq r\leq p$, the integrability of $F(\bar{X}_{0}, \ldots, \bar{X}_{M})$  and  $F(\bar{Y}_{0}, \ldots, \bar{Y}_{M})$ is guaranteed since under Assumption (I), the truncated Euler scheme has a finite $p$-th moments:
\begin{equation}\label{polymot}
\forall M\geq 1, \quad \Big\|  \sup_{0\leq m\leq M}|\widetilde X^M_{t_m}|\Big\|_p\leq C\big (1+\|X_0\|_p\big).
\end{equation}
where $C$ does not depend on $M$. This result is a direct consequence of~\cite[Lemma A.4]{liu2020functional} for the regular Euler scheme
\begin{equation}\label{polymomentdiscret}
\forall M\geq 1,  \quad 
\Big\Vert\sup_{0\leq m\leq M}\left| \bar{X}^{M}_{t_{m}}\right|\Big\Vert_{p}\leq C\big(1+\vertii{X_{0}}_{p}\big),
\end{equation}
as $|Z^h_m|\le |Z_m|, \;1\leq m\leq M$ by construction.


Now we introduce for every $m=0,\ldots, M-1,$  the Euler operators $\EMH$ in order to simplify the description of the regular or truncated Euler scheme
\begin{equation}\label{euleroperator}
\EMH(x, \mu, Z)\coloneqq x+h\cdot b(t_{m}, x, \mu)+\sqrt{h}\cdot \sigma(t_{m}, x, \mu)Z,
\end{equation}
where $\,x\in\RR$, $\mu\in\mathcal{P}_{1}(\RR)$ and $Z: \,(\Omega, \mathcal{F}, \mathbb{P})\rightarrow \big(\RR, Bor(\RR)\big)$ is a random variable. Then,
the truncated Euler scheme can be written by $\widetilde{X}_{m+1}=\EMH(\widetilde{X}_{m}, \widetilde{\mu}_{m}, Z^{h}_{m+1})$. For $f:\RR\rightarrow\RR$ a function with sub-exponential growth, let $\widetilde{P}_{m}^{\,h}$ denote the transition operator of the truncated Euler scheme defined by
\begin{align}\label{defP}
\widetilde{P}_{m+1}^{\,h}(f)(x, \mu)\coloneqq &\,\EE f\big(\EMH(x, \mu, Z_{m+1}^{h})\big)
=\,\EE f\big( x+h\cdot b(t_{m}, x, \mu)+\sqrt{h}\sigma(t_{m}, x, \mu)Z_{m+1}^{h}\big).
\end{align}
Remark that if $f$ has an $r$-polynomial growth, $1\leq r\leq p$, for a fixed $\mu\in\mathcal{P}_{1}(\RR)$, the function $x\mapsto \widetilde{P}_{m}^{\,h}(f)(x, \mu)$ has also an $r$-polynomial growth. 

Before proving Proposition~\ref{funmonoconvexEuler}, we first show in Section~\ref{Eulerpropa} that the truncated Euler scheme propagates the marginal monotone convex order.

\subsection{Marginal monotone convex order for the truncated Euler scheme}\label{Eulerpropa}

The main result of this section is the following proposition.
\begin{prop}\label{monconEuler}
Let $h\in (0, \frac{1}{2[b]_{\mathrm{Lip}_{x}}})$.  Let $(\widetilde{X}_{m})_{0\leq m\leq M}$ and $(\widetilde{Y}_{m})_{0\leq m\leq M}$ be the truncated Euler schemes defined by~(\ref{EulertruncatedX}) and~(\ref{EulertruncatedY}). 
\vspace{-0.3cm}\begin{enumerate}[$(a)$]
\item Under Assumption (I) and (II), we have 
$\;\widetilde{X}_{m}\mconright\widetilde{Y}_{m}, \quad m=0,\ldots, M.$
\item (Symmetric setting) {\color{black}Under Assumption (I) and (II'), we have 
\[\widetilde{Y}_{m}\mconright\widetilde{X}_{m}, \quad m=0,\ldots, M.\]}
\end{enumerate}\end{prop}

\vspace{-0.5cm}
Before we prove Proposition~\ref{monconEuler}, we first introduce  the following lemma, which is a direct application of~\cite[Lemma 3.2-(iii)]{jourdain2019convex}. 

\begin{lem}\label{conZtrun}
$(a)$ Let $Z: (\Omega, \mathcal{A}, \mathbb{P})\rightarrow \big(\RR, Bor(\RR)\big)$ be a symmetric integrable random variable (i.e. $Z$ has the same distribution of $-Z$). Let $h>0$ 
and let $Z^{h}\coloneqq T^{h}(Z)$. 
 If $\,0\leq u_1\leq u_2$, then $u_1 Z^{h}\conright u_2 Z^{h}$ in the sense that for every convex function $f: \RR\rightarrow\RR$, we have $\EE f(u_1 Z^{h})\leq \EE f(u_2 Z^{h})$.

\noindent $(b)$ Let $Z: (\Omega, \mathcal{A}, \mathbb{P})\rightarrow \big(\RR, Bor(\RR)\big)$ be a symmetric integrable random variable such that $\mathbb{E}\,e^{c\left|Z\right|}<+\infty$ for some $c>0$ and let $Z^{h}\coloneqq T^{h}(Z)$. 
Let $f: \RR\rightarrow\RR$ be a convex function. For every fixed $(t, x, \mu)\in [0,T]\times\RR\times\mathcal{P}_{1}(\RR)$ the function 
\[u\:\longmapsto \: \EE f(x + h\cdot b(t, x, \mu)+u Z^{h})\]
is non-decreasing and reaches its minimum in 0. 
\end{lem}
\begin{proof} (see~\cite{pages2016convex}) We only prove the part $(a)$ since the part $(b)$ is a direct application of $(a)$.
The random variable $Z^{h}$ is centered since it is bounded with a symmetric distribution. Then, if $f$ is convex, the function $u\mapsto \EE f(u Z^{h})$ is clearly convex and attains its minimum at $u=0$ by Jensen's inequality, so it is non-decreasing on $\RR_{+}$.
\end{proof}


The proof of Proposition~\ref{monconEuler} relies on the propagation of monotonicity and convexity by the operator $\widetilde{P}_{m}^{\,h}$ defined in~(\ref{defP}), which is established in the following lemma. Remark that in this  lemma, our discussion is based on a generalized symmetric random variable $Z$ satisfying $\EE e^{c\left|Z\right|}<+\infty$ for some $c>0$, instead of the Gaussian white noise $Z_m$ defined in~(\ref{Eulerclass}).  

\begin{lem}
\label{propamonoconv}
Let $h\in (0, \frac{1}{2[b]_{\mathrm{Lip}_{x}}})$. 
Let $Z: (\Omega, \mathcal{A}, \mathbb{P})\rightarrow \big(\RR, Bor(\RR)\big)$ be a symmetric random variable such that $\mathbb{E}\,e^{c\left|Z\right|}<+\infty$ and let $Z^{h}\coloneqq T^{h}(Z)$. \textcolor{black}{Assume conditions in Assumption (I) and (II)-(1),(2) are in force}. 
\begin{enumerate}[$(a)$]
\item Let $f:\RR\rightarrow\RR$ be a non-decreasing function with sub-exponential growth.  Then, for every $m=1,\ldots, M$ and for every fixed $\mu\in\mathcal{P}_{1}(\RR)$, the function 
\[x\longmapsto \mathbb{E}\, f\big(\EMH(x, \mu, Z^{h})\big)\]
is non-decreasing.
\item Let $f:\RR\rightarrow\RR$ be a non-decreasing convex function with sub-exponential growth. 
Then for every $m=1,\ldots, M$ and for every fixed $\mu\in\mathcal{P}_{1}(\RR)$, the function 
\[x\longmapsto \mathbb{E}\, f\big(\EMH(x, \mu, Z^{h})\big)\]
is non-decreasing and convex.
\end{enumerate}
\end{lem} 

\begin{proof}
$(a)$ As $\mu\in\mathcal{P}_{1}(\RR)$ is fixed, to alleviate notations, we denote $b_m(x)= b(t_m,x,\mu)$ and $\sigma_m(x)= \sigma(t_m,x,\mu)$, $0\leq m\leq M$. By Assumption (I), $b_m(x)$ and $\sigma_m(x)$ are Lipschitz continuous functions with respective Lipschitz constants $[b_m]_{\Lip}$ and $[\sigma_m]_{\Lip}$ satisfying $[b_m]_{\Lip}\leq L$ and $[\sigma_m]_{\Lip}\leq L$. 


\smallskip
\noindent {\sc Step~1} ($f$ {\em smooth}). Assume $f$ is also $C^1$.  
Then $f'\ge 0$ as $f$ is non-decreasing. Let $x$, $y\!\in \RR$, $x> y$. A first order Taylor expansion yields
\begin{align}
&\mathbb{E}\, f\big(\EMH(x, \mu, Z^{h})\big)-\mathbb{E}\, f\big(\EMH(y, \mu, Z^{h})\big)\nonumber\\
&=\EE\Big[\int_0^1f'\big(u\, {\cal E}_m^{h}(x,\mu,Z^h)+(1-u){\cal E}_m^{h}(y,\mu,Z^h)\big)du\nonumber\\
&\hspace{4cm}\cdot\Big( x-y +h\big(b_m(x)-b_m(y)\big) +\sqrt{h}\big( \sigma_m(x)-\sigma_m(y)\big)Z^h \Big) \Big].\nonumber
\end{align}
Moreover, 
\begin{align*}
x-y& +h\big(b_m(x)-b_m(y)\big) +\sqrt{h}\big( \sigma_m(x)-\sigma_m(y)\big)Z^h\\
&=(x-y)\Big(1+h \,\frac{b_m(x)-b_m(y)}{x-y}+\sqrt{h}\, \frac{\sigma_m(x)-\sigma_m(y)}{x-y}Z^h\Big)\\
& \ge (x-y)\Big(1-h[b_m]_{\text{Lip}}-\sqrt{h}\, [\sigma_{m}]_{\text{Lip}} |Z^h|\Big)\\
&\ge  (x-y)\Big(\tfrac12-\sqrt{h}\, [\sigma_{m}]_{\text{Lip}} |Z^h|\Big)\ge 0
\end{align*}
owing to the definition of $Z^{h}$. 
As $f$ is non-decreasing,   $f'\geq0$. Hence $\mathbb{E}\, f\big(\EMH(x, \mu, Z^{h})\big)-\mathbb{E}\, f\big(\EMH(y, \mu, Z^{h})\big) \ge 0$.

\smallskip
\noindent {\sc Step~2} ({\em Regularization}.) Assume now $f$ is simply non-decreasing (but still has a sub-exponential growth of the form $|f(x)|\le Ke^{\kappa|x|}$). Let $\zeta$ be a random variable having probability distribution $\mathcal{N}(0, 1)$ independent of $Z$. 
For every $\varepsilon>0$, let $f_{\varepsilon}(x)= \EE\, f(x+\sqrt{\varepsilon}\zeta)$. The functions $f_{\varepsilon}$ is well defined since \[|f_{\varepsilon}(x)| = \EE\, f(x+\sqrt{\varepsilon}\zeta)\le  K \EE\, e^{\kappa |x+\sqrt{\varepsilon}\zeta|}\le Ke^{\kappa |x|} \EE\, e^{\kappa \sqrt{\varepsilon}|\zeta|}  \le 2K e^{\frac{1}{2}\kappa^{2} \varepsilon}e^{\kappa|x|}<+\infty\]
and the functions is clearly non-decreasing. 



One checks that, under this sub-exponential growth assumption of $f$,  the function $f_{\varepsilon}$ is differentiable  with derivative $f'_{\varepsilon}(x)=\frac{1}{\sqrt{\varepsilon}} \EE\,\big( f(x+\sqrt{\varepsilon}\zeta)\zeta\big)$. 
Moreover, $\zeta$ and $Z$ being independent,
\begin{align*}
\EE f_{\varepsilon}\big(\mathcal{E}_m^h(x, \mu, Z^{h})\big)&= \EE f\big( x +h b(t_k,x,\mu) +\sqrt{h}\sigma(t_k,x,\mu)Z^h+\sqrt{\varepsilon}\,\zeta \big),
\end{align*}
so that if $\varepsilon \!\in (0, \varepsilon_0]$ for some fixed $\varepsilon_0>0$,
\[
|f \big( x +h b(t_k,x,\mu) +\sqrt{h}\sigma(t_k,x,\mu)Z^h+\sqrt{\varepsilon}\zeta \big)|\le C_{b,\sigma,h,x}e^{c(\sqrt{h}|\sigma(t_k,x,\mu)||Z| +\varepsilon_0 |\zeta|)}\!\in L^1(\mathbb{P}).
\]
It follows from Lebesgue's dominated convergence theorem that  $\EE f_{\varepsilon}\big(\mathcal{E}_m^h(x, \mu, Z^{h})\big)$ converges to $\EE f\big(\mathcal{E}_m^h(x, \mu, Z^{h})\big)$. Now the function $x\mapsto \EE f_{\varepsilon}\big(\mathcal{E}_m^h(x, \mu, Z^{h})\big)$ is non-decreasing owing to Step~1 which in turn implies that so is $x\mapsto \EE f\big(\mathcal{E}_m^h(x, \mu, Z^{h})\big)$.

\noindent $(b)$ By applying $(a)$, the function $x\longmapsto \mathbb{E}\, f\big(\EMH(x, \mu, Z^{h})\big)$ is non-decreasing. Now we prove its convexity. Let $x, y \in \RR$ and $\lambda\in[0, 1]$.
\begin{align}
&\EE f\big(\EMH (\lambda x+(1-\lambda)y, \mu, Z^h)\big)\nonumber\\
&= \EE f \big(\lambda x+(1-\lambda)y+h\cdot b(t_m, \lambda x+(1-\lambda)y, \mu)+\sqrt{h}\cdot \sigma(t_m, \lambda x+(1-\lambda)y, \mu)Z^h\big)\nonumber\\
&{\color{black}= \EE f \big(\lambda x+(1-\lambda)y+h\cdot b(t_m, \lambda x+(1-\lambda)y, \mu)+\sqrt{h}\cdot \left|\sigma(t_m, \lambda x+(1-\lambda)y, \mu)\right|Z^h\big)}\nonumber\\
&{\color{black}\qquad\text{(as the random variable $Z^{h}$ is symmetric)}}\nonumber\\
&\leq \EE f \big(\lambda x+(1-\lambda)y+\lambda h\cdot b(t_m, x, \mu)+(1-\lambda) h\cdot b(t_m, y, \mu)+\sqrt{h}\cdot {\color{black}\left|\sigma(t_m, \lambda x+(1-\lambda)y, \mu)\right|}\,Z^h\big)\nonumber\\
&\qquad\text{(by the convex assumption on $b$ and the monotonicity of $f$)}\nonumber\\
&\leq \EE f \big(\lambda x+(1-\lambda)y+\lambda h\cdot b(t_m, x, \mu)+(1-\lambda) h\cdot b(t_m, y, \mu)+{\color{black}\lambda\sqrt{h}\cdot \left|\sigma(t_m,  x, \mu)\right|Z^h }\nonumber\\
&\qquad{+\color{black} (1-\lambda)\sqrt{h}\cdot \left|\sigma(t_m, y, \mu)\right| Z^h}\big) \quad\text{(by the convex assumption on $\left|\sigma\right|$ and Lemma~\ref{conZtrun}-$(b)$)}\nonumber\\
&\leq \lambda\EE f \big( x+ h\cdot b(t_m, x, \mu){+\color{black}\lambda\sqrt{h}\cdot \left|\sigma(t_m,  x, \mu)\right|Z^h \big)+(1-\lambda)\EE f \big( y+ h\cdot b(t_m, y, \mu)+ \sqrt{h}\cdot \left|\sigma(t_m, y, \mu)\right| Z^h}\big) \nonumber\\
&\qquad\text{(by the convexity of $f$)}\nonumber\\
&\leq \lambda \EE f\big(\EMH(x, \mu, Z^{h})\big)+(1-\lambda)\EE f\big(\EMH(y, \mu, Z^{h})\big) \qquad\text{(by the symmetry of $Z^{h}$)}.\hfill\qedhere\nonumber
\end{align}
\end{proof}

For every $m=0,\ldots, M$, let $\mathcal{F}_{m}$ denote the $\sigma$-algebra generated by $(X_0, Y_0, Z_1,\ldots, Z_m)$. Now we prove Proposition~\ref{monconEuler}.

\begin{proof}[Proof of Proposition~\ref{monconEuler}]$(a)$ \textcolor{black}{We proceed by induction.}
Assumption (II)-(4) directly implies $\widetilde{X}_{0}\mconright \widetilde{Y}_{0}$.  
Assume now that $\widetilde{X}_{m}\mconright \widetilde{Y}_{m}$.  \textcolor{black}{Let $\varphi: \RR \to \mathbb{R}$ be a non-decreasing convex function with linear growth}.
\begin{align}
\EE \big[\varphi(\widetilde{X}_{m+1})\big] &= \EE\Big[ \varphi\big( \widetilde{X}_{m}+h\cdot b(t_{m}, \widetilde{X}_{m}, \widetilde{\mu}_{m})+\sqrt{h}\cdot \sigma(t_{m}, \widetilde{X}_{m}, \widetilde{\mu}_{m})Z^{h}_{m+1}\big)\Big]\nonumber\\
&= \EE \Big[\EE \big[\varphi\big( \widetilde{X}_{m}+h\cdot b(t_{m}, \widetilde{X}_{m}, \widetilde{\mu}_{m})+\sqrt{h}\cdot \sigma(t_{m}, \widetilde{X}_{m}, \widetilde{\mu}_{m})Z^{h}_{m+1}\big)\mid \mathcal{F}_{m}   \big]\Big]\nonumber\\
&=\int_{\mathbb{R}}\widetilde{\mu}_{m}(dx)\,\EE \Big[ \varphi \big( x+h\cdot b(t_{m}, x,\, \widetilde{\mu}_{m})+\sqrt{h}\cdot \sigma(t_{m}, x, \, \widetilde{\mu}_{m})Z^{h}_{m+1}\big)\Big]\nonumber\\
&\hspace{0.5cm}\text{(the integrability is due to \eqref{polymot} as $\varphi$ has  linear growth)}\nonumber\\
&{\color{black}=\int_{\mathbb{R}}\widetilde{\mu}_{m}(dx)\,\EE \Big[ \varphi \big( x+h\cdot b(t_{m}, x,\, \widetilde{\mu}_{m})+\sqrt{h}\cdot \left|\sigma(t_{m}, x, \, \widetilde{\mu}_{m})\right|Z^{h}_{m+1}\big)\Big]}\nonumber\\
&\leq \int_{\mathbb{R}^{d}}\widetilde{\mu}_{m}(dx)\,\EE \Big[ \varphi \big( x+h\cdot b(t_{m}, x,\, \widetilde{\nu}_{m})+\sqrt{h}\cdot {\color{black}\left|\sigma(t_{m}, x, \, \widetilde{\nu}_{m})\right|}Z^{h}_{m+1}\big)\Big]\nonumber\\
&\hspace{0.5cm}(\text{by Assumption (II)-(2) and Lemma~\ref{conZtrun}-$(b)$})\nonumber\\
&\leq \int_{\mathbb{R}^{d}}\widetilde{\nu}_{m}(dx)\,\EE \Big[ \varphi \big( x+h\cdot b(t_{m}, x,\, \widetilde{\nu}_{m})+\sqrt{h}\cdot {\color{black}\left|\sigma(t_{m}, x, \, \widetilde{\nu}_{m})\right|}Z^{h}_{m+1}\big)\Big]\nonumber\\
&\hspace{0.5cm}(\text{by Lemma~\ref{propamonoconv}})\nonumber\\
&\leq \int_{\mathbb{R}^{d}}\widetilde{\nu}_{m}(dx)\,\EE \Big[ \varphi \big( x+h\cdot \beta(t_{m}, x,\, \widetilde{\nu}_{m})+\sqrt{h}\cdot {\color{black}\left|\theta(t_{m}, x, \, \widetilde{\nu}_{m})\right| }Z^{h}_{m+1}\big)\Big]\nonumber\\
&\hspace{0.5cm}(\text{by Assumption (II)-(3) and Lemma~\ref{conZtrun}-$(b)$})\nonumber\\
&= \int_{\mathbb{R}^{d}}\widetilde{\nu}_{m}(dx)\,\EE \Big[ \varphi \big( x+h\cdot \beta(t_{m}, x,\, \widetilde{\nu}_{m})+\sqrt{h}\cdot {\color{black}\theta(t_{m}, x, \, \widetilde{\nu}_{m})}Z^{h}_{m+1}\big)\Big]\nonumber\\
&=\EE [\varphi(\widetilde{Y}_{m+1})]. \label{reverse}
\end{align}
Thus $\widetilde{X}_{m+1}\mconright \widetilde{Y}_{m+1}$ by applying Lemma~\ref{onlylineargrowth} and one concludes by a forward induction. 

\noindent\textcolor{black}{ $(b)$ The proof of $(b)$ is very similar to $(a)$. First, Assumption (II)- $(4')$ directly implies $\widetilde{Y}_{0} \mconright\widetilde{X}_{0}$. Then the rest part of the induction is simply a reversion of the three inequalities in (\ref{reverse}).  }
\end{proof}


\subsection{Functional monotone convex order for the truncated Euler scheme}\label{funcmcvEuler}
This section is devoted to the proof of Proposition~\ref{funmonoconvexEuler}. 
For any $K\in\mathbb{N}^{*}$, we consider the norm on $\mathbb{R}^{K}$ defined by $\vertii{x}\coloneqq \sup_{1\leq i \leq K}\left|x_{i}\right|$ for every $x=(x_{1},\ldots, x_{K})\in\mathbb{R}^{K}$. 
For any $m_{1}, m_{2}\in\mathbb{N}^{*}$ with $m_{1}\leq m_{2}$, we denote $x_{m_1:m_{2}}\coloneqq (x_{m_{1}}, x_{m_{1}+1}, \ldots, x_{m_{2}})\in\RR^{m_{2}-m_{1}+1}$. Similarly, we denote $\mu_{m_{1}: \;m_{2}}\coloneqq(\mu_{m_{1}}, \ldots, \mu_{m_{2}})\in\big(\mathcal{P}_{1}(\RR)\big)^{m_{2}-m_{1}+1}$.
We recursively define a sequence of functions \[\Phi_{m}: \RR^{m+1}\times \big(\mathcal{P}_{1}(\RR)\big)^{M-m+1}\rightarrow \RR, \quad m=0, \ldots, M\] 
in a backward way as follows: 

$\blacktriangleright\quad$Set 
\begin{equation}\label{defPhi1}
\Phi_{M}(x_{0: M}, \mu_{M})\coloneqq F(x_{0}, \ldots, x_{M})
\end{equation} 
with the same $F$ as in Proposition~\ref{funmonoconvexEuler}. 

$\blacktriangleright\quad$For $m=0, \ldots, M-1$, set
\begin{align}
&\Phi_{m}(x_{0: m}, \mu_{m: M}) \coloneqq\big(\widetilde{P}^{\,h}_{m+1}\Phi_{m+1}(x_{0: m}, \,\cdot\,, \mu_{m+1: M})\big)(x_{m}, \mu_{m})\nonumber\\
&\hspace{0.7cm}=\EE \Big[\Phi_{m+1}\big(x_{0:m}, \mathcal{E}_{m}^{h}(x_{m}, \mu_{m}, Z_{m+1}^{h}), \mu_{m+1:M}\big)\Big]&\nonumber\\
\label{defPhi2}
&\hspace{0.7cm}=\EE \Big[\Phi_{m+1}\big(x_{0: m}, x_m+h\cdot b(t_m, x_m, \mu_m)+\sqrt{h}\cdot \sigma(t_m, x_m, \mu_m)Z^{h}_{m+1}, \mu_{m+1: M}\big)\Big].&
\end{align}
The functions  $\Phi_{m}, m=0, \ldots, M$ 
 share the following properties. 
\begin{lem}\label{propphi}
Let $h\in (0, \frac{1}{2[b]_{\mathrm{Lip}_{x}}})$.  Under Assumption (I) and (II)-(1), (2), for every $m=0, \ldots, M$,  

\vspace{-0.3cm}
\begin{enumerate}[$(i)$]
\item for a fixed $\mu_{m:M}\in\big(\mathcal{P}_{1}(\RR)\big)^{M-m+1}$, the function $\Phi_{m}(\;\cdot\;, \mu_{m: M})$ is convex and non-decreasing in $x_{0:m}$ w.r.t the partial order defined in~(\ref{partialorder}) and has an $r$-polynomial growth in $x_{0: m}$, so that $\Phi_{m}$ is well-defined. 
\item for a fixed $x_{0:m}\in\RR^{m+1}$, the function $\Phi_{m}(x_{0: m}, \;\cdot\;)$ is non-decreasing in $\mu_{m:M}$ with respect to the monotone convex order  in the sense that for any $\mu_{m:M}, \nu_{m:M}\in \big(\mathcal{P}_{1}(\RR)\big)^{M-m+1}$ such that $\mu_{i}\mconright \nu_{i}, i= m, \ldots, M$, 
\begin{align}
& \Phi_{m}(x_{0:m}, \mu_{m: M})\leq \Phi_{m}(x_{0:m}, \nu_{m: M}). 
\end{align}
\smallskip
\end{enumerate}
\end{lem}

\vspace{-1.2cm}
\begin{proof} 
\noindent  $(i)$ The function $\Phi_{M}(\cdot, \mu_{M})$ is convex and non-decreasing w.r.t. the partial order~(\ref{partialorder}) in $x_{0:M}$ owing to the hypothesis on $F$. Now assume that $x_{0: m+1}\mapsto\Phi_{m+1}(x_{0:m+1}, \mu_{m+1:M})$ is convex and non-decreasing w.r.t. the partial order~(\ref{partialorder}). 

\noindent {\sc Step~1} ($\Phi_{m}$ is convex). For any $x_{0:m}, y_{0:m}\in \RR^{m+1}$ and $\lambda\in[0, 1]$, it follows that 
\begin{align}
&\Phi_{m}\big(\lambda x_{0:m}+(1-\lambda)y_{0:m}, \mu_{m:M}\big)\nonumber\\
&\hspace{0.5cm}=\EE \Phi_{m+1}\Big( \lambda x_{0:m}+(1-\lambda)y_{0:m}, \lambda x_m+(1-\lambda)y_m+h\cdot b\big(t_m, \lambda x_m+(1-\lambda)y_{m}, \mu_m\big)\nonumber\\
&\hspace{3cm}+\sqrt{h}\cdot\sigma\big(t_m, \lambda x_m+(1-\lambda)y_{m}, \mu_m\big) Z_{m+1}^{h}, 
\mu_{m+1:M}\Big)\nonumber\\
&\hspace{0.5cm}\leq\EE \Phi_{m+1}\Big( \lambda x_{0:m}+(1-\lambda)y_{0:m}, \lambda x_m+(1-\lambda)y_m+\lambda h\cdot b\big(t_m,  x_m, \mu_m\big)+(1-\lambda)h\cdot b\big(t_m,  y_{m}, \mu_m\big)\nonumber\\
&\hspace{3cm}+\sqrt{h}\cdot\sigma\big(t_m, \lambda x_m+(1-\lambda)y_{m}, \mu_m\big) Z_{m+1}^{h}, 
\mu_{m+1:M}\Big)\nonumber\\
&\hspace{1cm} \text{(by Assumption (II)-(1) and the monotonicity of $\Phi_{m+1}$)}\nonumber\\
&\hspace{0.5cm}=\EE \Phi_{m+1}\Big( \lambda x_{0:m}+(1-\lambda)y_{0:m}, \lambda x_m+(1-\lambda)y_m+\lambda h\cdot b\big(t_m,  x_m, \mu_m\big)+(1-\lambda)h\cdot b\big(t_m,  y_{m}, \mu_m\big)\nonumber\\
&\hspace{3cm}+\sqrt{h}\cdot{\color{black}\left|\sigma\big(t_m, \lambda x_m+(1-\lambda)y_{m}, \mu_m\big) \right|}Z_{m+1}^{h}, 
\mu_{m+1:M}\Big)\nonumber\\
&\hspace{0.5cm}\leq\EE \Phi_{m+1}\Big( \lambda x_{0:m}+(1-\lambda)y_{0:m}, \lambda x_m+(1-\lambda)y_m+\lambda h\cdot b\big(t_m,  x_m, \mu_m\big)+(1-\lambda)h\cdot b\big(t_m,  y_{m}, \mu_m\big)\nonumber\\
&\hspace{3cm}+{\color{black}\left|\lambda\sqrt{h}\cdot\sigma\big(t_m, x_m, \mu_m\big) +(1-\lambda)\sqrt{h}\cdot \sigma\big(t_m, y_{m}, \mu_m\big)\right| }Z_{m+1}^{h}, 
\mu_{m+1:M}\Big)\nonumber\\
&\hspace{1cm} \text{(by Assumption (II)-(1) and Lemma~\ref{conZtrun}-$(b)$)}\nonumber\\
&\hspace{0.5cm}\leq\EE \Phi_{m+1}\Big( \lambda x_{0:m}+(1-\lambda)y_{0:m}, \lambda x_m+(1-\lambda)y_m+\lambda h\cdot b\big(t_m,  x_m, \mu_m\big)+(1-\lambda)h\cdot b\big(t_m,  y_{m}, \mu_m\big)\nonumber\\
&\hspace{3cm}+\lambda\sqrt{h}\cdot{\color{black}\left|\sigma\big(t_m, x_m, \mu_m\big) \right|}Z_{m+1}^{h}+(1-\lambda)\sqrt{h}\cdot {\color{black}\left|\sigma\big(t_m, y_{m}, \mu_m\big)\right|} Z_{m+1}^{h}, 
\mu_{m+1:M}\Big)\nonumber\\
&\hspace{0.5cm}=\EE \Phi_{m+1}\Big( \lambda x_{0:m}+(1-\lambda)y_{0:m}, \lambda \mathcal{E}_{h}(x_m, \mu_m, Z_{m+1}^{h})+ (1 - \lambda)\mathcal{E}_{h}(y_m, \mu_m, Z_{m+1}^{h}), \mu_{m+1:M}\Big)\nonumber\\
&{\color{black}\hspace{1cm} \text{(by the symmetry of the distribution of $Z_{m+1}^{h}$)}}\nonumber\\
&\hspace{0.5cm}\leq \lambda \EE \Phi_{m+1}\Big( x_{0:m}, \mathcal{E}_{h}(x_m, \mu_m, Z_{m+1}^{h}), \mu_{m+1:M}\Big)+(1-\lambda)\EE \Phi_{m+1}\Big( y_{0:m}, \mathcal{E}_{h}(y_m, \mu_m, Z_{m+1}^{h}), \mu_{m+1:M}\Big)\nonumber\\
&\hspace{1cm} \text{(by the convexity of $\Phi_{m+1}$)}\nonumber\\
&\hspace{0.5cm}=\lambda \Phi_{m}(x_{0:m}, \mu_{m:M})+(1-\lambda) \Phi_{m}(y_{0:m}, \mu_{m:M}).\nonumber
\end{align}
Thus the function $\Phi_{m}(\,\cdot\,, \mu_{m:M})$ is convex and one concludes by a backward induction. 

\noindent {\sc Step~2} \big($\Phi_{m}$ is non-decreasing w.r.t. the partial order~(\ref{partialorder})\big). 
For any $x_{0:m}, z_{0:m}\in \RR^{m+1}$ such that $x_{0:m}\preceq z_{0:m}$. 
\begin{align}
\Phi_{m}&(x_{0:m}, \mu_{m:M})=\EE\Phi_{m+1}\big(x_{0:m}, \mathcal{E}_{h}(x_{m}, \mu_{m}, Z_{m+1}^{h}), \mu_{m+1:M}\big)\nonumber\\
&\leq \EE\Phi_{m+1}\big(z_{0:m}, \mathcal{E}_{h}(x_{m}, \mu_{m}, Z_{m+1}^{h}), \mu_{m+1:M}\big)\nonumber\\
&\hspace{0.5cm}\text{(by the monotonicity of $\Phi_{m+1}$)}\nonumber\\
&\leq \EE\Phi_{m+1}\big(z_{0:m}, \mathcal{E}_{h}(z_{m}, \mu_{m}, Z_{m+1}^{h}), \mu_{m+1:M}\big)=\Phi_{m}(z_{0:m}, \mu_{m:M}),\nonumber\\
\end{align}
where the last inequality is due to Lemma~\ref{propamonoconv} as the function $x\mapsto \Phi_{m+1}(z_{0:m}, x, \mu_{m+1:M})$ is a convex non-decreasing function. 
Thus one concludes by a backward induction. Moreover, it is obvious by a backward induction that the functions $\Phi_{m}, 1\leq m\leq M,$ have an $r$-polynomial growth by Assumption (I) and the assumption made on $F$.



\noindent$(ii)$ Firstly, it is obvious that for any $\mu_{M}, \nu_{M}\in\mathcal{P}_{1}(\RR)$ such that $\mu_{M}\mconright \nu_{M}$, we have
\[\Phi_{M}(x_{0:M}, \mu_{M})=F(x_{0:M})=\Phi_{M}(x_{0:M}, \nu_{M}).\]
Assume that $\Phi_{m+1}(x_{0:m+1}, \cdot\,)$ is non-decreasing in $\mu_{m+1:M}$ with respect to the monotone convex order. For any $\mu_{m:M}, \nu_{m:M}\in\big(\mathcal{P}_{1}(\RR)\big)^{M-m+1}$ such that $\mu_{i}\mconright\nu_{i}, i=m, \ldots, M,$ we have
\begin{align}
&\Phi_{m}(x_{0:m}, \mu_{m:M})\nonumber\\
&\hspace{0.5cm}=\EE \Big[ \Phi_{m+1}\big(x_{0:m}, x_{m}+hb(t_{m}, x_{m}, \mu_{m})+\sqrt{h}\sigma(t_{m}, x_{m}, \mu_{m})Z_{m+1}^{h}, \mu_{m+1: M}\big)\Big]\nonumber\\
&\hspace{0.5cm}=\EE \Big[ \Phi_{m+1}\big(x_{0:m}, x_{m}+hb(t_{m}, x_{m}, \mu_{m})+\sqrt{h}{\color{black}\,\left|\sigma(t_{m}, x_{m}, \mu_{m})\right|}Z_{m+1}^{h}, \mu_{m+1: M}\big)\Big]\nonumber\\
&\hspace{0.5cm}\leq\EE \Big[ \Phi_{m+1}\big(x_{0:m}, x_{m}+hb(t_{m}, x_{m}, \nu_{m})+\sqrt{h}{\color{black}\,\left|\sigma(t_{m}, x_{m}, \nu_{m})\right|}Z_{m+1}^{h}, \mu_{m+1: M}\big)\Big]\nonumber\\
&\hspace{1cm}\text{(by Assumption~(II)-(2) and Lemma~\ref{conZtrun}-$(b)$ since $\Phi{_{m+1}} (x_{0:m}, \cdot, \mu_{m+1:M})$ is convex and non-decreasing)}\nonumber\\
&\hspace{0.5cm}\leq\EE \Big[ \Phi_{m+1}\big(x_{0:m}, x_{m}+hb(t_{m}, x_{m}, \nu_{m})+\sqrt{h}\sigma(t_{m}, x_{m}, \nu_{m})Z_{m+1}^{h}, \nu_{m+1: M}\big)\Big]\nonumber\\
&\hspace{1cm}\text{(as $\Phi_{m+1}$ is non-decreasing w.r.t the monotone convex order)}\nonumber\\
&\hspace{0.5cm}=\Phi_{m}(x_{0:m}, \nu_{m:M}).\nonumber
\end{align}
Then one concludes by a backward induction. 
\end{proof}

Similarly, we define $\Psi_{m}:\RR^{m+1}\times \big(\mathcal{P}_{1}(\RR)\big)^{M-m+1}\rightarrow \RR,  \,m=0, \ldots, M$ by
\begin{flalign}
&\Psi_{M}(x_{0:M}, \mu_{M})\hspace{0.25cm}\coloneqq F(x_{0:M}),&\nonumber\\
\label{defpsi}
&\Psi_{m}(x_{0:m}, \mu_{m:M})\coloneqq 
\EE \Big[\Psi_{m+1}\big(x_{0:m}, x_{m}+h\beta(t_m, x_m, \mu_m)+\sqrt{h}\theta(t_m, x_{m}, \mu_{m})Z^{h}_{m+1}, \mu_{m+1:M}\big)\Big]. &
\end{flalign}
Recall the notation $\widetilde{\mu}_{m}\coloneqq \mathbb{P}_{\widetilde{X}_{m}}$ and $\widetilde{\nu}_{m}\coloneqq \mathbb{P}_{\widetilde{Y}_{m}}$. By applying the same recursion as in~\cite[Lemma 2.6]{liu2020functional}, we know that for every $m=0,\ldots, M$, 
\begin{align}\label{especonditionel}
&\Psi_{m}(\widetilde{X}_{0:m}, \widetilde{\mu}_{m:M})=\EE \big[F(\widetilde{X}_{0}, \ldots, \widetilde{X}_{M})\mid \mathcal{F}_{m}\big]\:\; \text{and}\:\;\Psi_{m}(\widetilde{Y}_{0:m}, \widetilde{\nu}_{m:M})=\EE \big[F(\widetilde{Y}_{0}, \ldots, \widetilde{Y}_{M})\mid \mathcal{F}_{m}\big].
\end{align}


\begin{proof}[Proof of Proposition~\ref{funmonoconvexEuler}]
\noindent $(a)$ 
We first prove by a backward induction that for every $m=0, \ldots, M$, $\Phi_{m}\leq\Psi_{m}$. 

It follows from the definition of $\Phi_{M}$ and $\Psi_{M}$ that $\Phi_{M}=\Psi_{M}$. Assume now $\Phi_{m+1}\leq \Psi_{m+1}$. For any $x_{0:m}\in\RR^{m+1}$ and $\mu_{m:M}\in\big(\mathcal{P}_{1}(\RR)\big)^{M-m+1}$, we have 
\begin{align}
&\Phi_{m}(x_{0:m}, \mu_{m:M})\nonumber\\
&\hspace{0.5cm}=\EE\big[\Phi_{m+1}\big(x_{0:m}, x_{m}+hb(t_{m}, x_{m}, \mu_{m})+\sqrt{h}\sigma(t_{m}, x_{m}, \mu_{m})Z_{m+1}^{h}, \mu_{m+1:M}\big)\big]\nonumber\\
&\hspace{0.5cm}=\EE\big[\Phi_{m+1}\big(x_{0:m}, x_{m}+hb(t_{m}, x_{m}, \mu_{m})+\sqrt{h}{\color{black}\,\left|\sigma(t_{m}, x_{m}, \mu_{m})\right|\,}Z_{m+1}^{h}, \mu_{m+1:M}\big)\big]\nonumber\\
&\hspace{0.5cm}\leq\EE\big[\Phi_{m+1}\big(x_{0:m}, x_{m}+h\beta(t_{m}, x_{m}, \mu_{m})+\sqrt{h}{\color{black}\,\left|\theta(t_{m}, x_{m}, \mu_{m})\right|}Z_{m+1}^{h}, \mu_{m+1:M}\big)\big]\nonumber\\
&\hspace{0.5cm}\quad\text{(by Assumption~(II)-(3) and Lemma~\ref{conZtrun}-$(b)$,~\ref{propphi})}\nonumber\\
&\hspace{0.5cm}=\EE\big[\Phi_{m+1}\big(x_{0:m}, x_{m}+h\beta(t_{m}, x_{m}, \mu_{m})+\sqrt{h}\,\theta(t_{m}, x_{m}, \mu_{m})Z_{m+1}^{h}, \mu_{m+1:M}\big)\big]\nonumber\\
&\hspace{0.5cm}\quad\text{\color{black}(by the symmetry of the distribution of $Z_{m+1}^{h}$)}\nonumber\\
&\hspace{0.5cm}\leq\EE\big[\Psi_{m+1}\big(x_{0:m}, x_{m}+h\beta(t_{m}, x_{m}, \mu_{m})+\sqrt{h}\theta(t_{m}, x_{m}, \mu_{m})Z_{m+1}^{h}, \mu_{m+1:M}\big)\big]\nonumber\\
&\hspace{0.5cm}=\Psi_{m}(x_{0:m}, \mu_{m:M}).\label{backindu}
\end{align}
Thus,  the backward induction is completed and 
\begin{equation}\label{phipsiorder}
\forall \, m=0, \ldots, M, \quad \Phi_{m}\leq\Psi_{m}.
\end{equation}
Consequently, 
\begin{align}
\EE\big[ F(\widetilde{X}_{0}, \ldots, \widetilde{X}_{M})\big]&=\EE \Phi_{0}(\widetilde{X}_{0}, \widetilde{\mu}_{0:M})\hspace{0.9cm}\text{(by~(\ref{especonditionel}))}\nonumber\\
&\leq\EE \Phi_{0}(\widetilde{Y}_{0}, \widetilde{\mu}_{0:M})\hspace{1cm}\text{(by Lemma~\ref{propphi}-$(i)$ since } \widetilde{X}_{0}\mconright\widetilde{Y}_{0})\nonumber\\
&\leq \EE \Phi_{0}(\widetilde{Y}_{0}, \widetilde{\nu}_{0:M})  \hspace{1cm}\,\text{(by Lemma~\ref{propphi}-$(ii)$ and Proposition~\ref{monconEuler}-$(a)$)}\nonumber\\
&\leq \EE \Psi_{0}(\widetilde{Y}_{0}, \widetilde{\nu}_{0:M}) \hspace{1cm}\,\text{(by~(\ref{phipsiorder}))}\nonumber\\
&=\EE \big[F(\widetilde{Y}_{0}, \ldots, \widetilde{Y}_{M})\big].\label{inve2}
\end{align}

\noindent \textcolor{black}{$(b)$ Under Assumption (I) and (II'), one can prove $\Phi_m\geq \Psi_m, \;m=0, \ldots, M$ by using the same backward induction as in (\ref{backindu}) by applying conditions in Assumption (II')$-(3')$ instead of Assumption (II)$-(3)$. Moreover, Proposition \ref{monconEuler}-$(b)$ implies that under the same assumptions, $\widetilde{Y}_{m}\mconright\widetilde{X}_{m},$ $m=0,\ldots, M$. Thus, one can conclude the proof by  simply reversing the three inequalities in (\ref{inve2}).}
\end{proof}


\section{Monotone convex order for the McKean-Vlasov process}

The main goal of this section is to prove Theorem~\ref{main}. The key step from Proposition~\ref{funmonoconvexEuler} to Theorem~\ref{main} is the convergence of the truncated Euler scheme, proved in the next section. 

\subsection{Convergence of the truncated Euler scheme}
This section is devoted to prove the convergence of the truncated Euler scheme~(\ref{EulertruncatedX}) and~(\ref{EulertruncatedY}) to the unique solution of the equations~(\ref{defx}) and~(\ref{defy}). We will state this convergence only for $(X_{t})_{t\in[0, T]}$ but the proof remains true for $(Y_{t})_{t\in[0, T]}$ as well. We first recall several results in~\cite{liu2020functional} for the convergence of the regular Euler scheme~(\ref{Eulerclass}). 

\begin{prop}\label{convEulerclass}
Let $(X_{t})_{t\in[0, T]}$ be the unique strong solution of~(\ref{defx}) and let $(\bar{X}_{t})_{t\in[0, T]}$ be the continuous Euler scheme defined by: $\bar{X}_{0}^{M}=X_0$ and
\begin{equation}\label{continuousEulerclass}
\forall t\in [t_m, t_m+1), \quad \bar{X}^{M}_{t}\coloneqq  \bar{X}^{M}_{t_{m}}+ b(t_{m},  \bar{X}^{M}_{t_m}, \bar{\mu}^{M}_{t_m})(t-t_m)+ \sigma(t_{m},  \bar{X}^{M}_{t_m}, \bar{\mu}^{M}_{t_m})(B_t - B_{t_m})
\end{equation}
where $\bar{\mu}_{t_m}^{M}$ denotes the probability distribution of $\bar{X}_{t_m}^{M}$. 
Under Assumption (I), 
\begin{enumerate}[$(a)$]
\item there exists a constant $C_{p, b, \sigma}$ depending on $p, b, \sigma$ such that for every $t\in[0, T]$
\begin{equation}\label{polymoment}\forall M\geq 1,  \quad \vertii{\sup_{u\in[0, t]}\left| X_u\right|}_{p}\vee \vertii{\sup_{u\in[0, t]}\left| \bar{X}^{M}_{u}\right|}_{p}\leq C_{p, b, \sigma} e^{C_{p, b, \sigma}t}\big(1+\vertii{X_{0}}_{p}\big);
\end{equation}
\item there exists a constant $\kappa$ depending on $L, b, \sigma, \vertii{X_{0}}_{p}, p, T$ such that for any $s, t\in[0, T], \,s<t$,
\[\forall\, M\geq1, \quad\quad\quad \vertii{\bar{X}_{t}^{M}-\bar{X}_{s}^{M}}_{p}\leq \kappa \sqrt{t-s}; \]
\item there exists a constant $C$ depending on $L, b,\sigma,p, \vertii{X_0}_p, T, \gamma$ such that 
\[\vertii{\sup_{t\in [0, T]}\left| X_{t}-\bar{X}_{t}^{M}\right|}_{p}\leq C h^{\frac{1}{2}\wedge\gamma}.\]
\end{enumerate}

\end{prop}

As~(\ref{continuousEulerclass}), we also need a ``continuous'' version of the truncated Euler scheme~(\ref{EulertruncatedX}). For this purpose, we first define the following interpolator. 
Recall that $t_{m}=m\cdot \frac{T}{M}, m=0,  \ldots, M$. 
\begin{defn}\label{interpolator}
For every integer $M\geq1$, we define the piecewise affine interpolator $i_{M}: x_{0:M}\in\RR^{M+1}\mapsto i_{M}(x_{0:M})\in\mathcal{C}([0, T], \RR)$ by 
\begin{align}
& \forall\, m=0, \ldots, M-1, \;\forall\, t\in[t_{m}, t_{m+1}],\hspace{0,5cm}i_{M}(x_{0:M})(t)=\frac{M}{T}\big[(t_{m+1}-t)x_{m}+(t-t_{m})x_{m+1}\big].&\nonumber
\end{align}
\end{defn}

\vspace{-0.2cm}
\noindent Now we define the {\em piecewise affine} truncated Euler scheme, denoted by $(\widetilde X^M_{t})_{t\in[0, T]}$, as follows:
\begin{equation}\label{affineEuler} 
\forall\, 0\leq m\leq M-1, \:\forall\, t\!\in[t_m,t_{m+1}], \quad \widetilde X^M_{t}\coloneqq i_{M}(\widetilde X_{0:M})(t)= \frac{t_{m+1}-t}{t_{m+1}-t_m}\widetilde X_{t_m}+   \frac{t-t_{m} }{t_{m+1}-t_m}\widetilde X_{t_{m+1}}.
\end{equation}


\begin{prop}[Convergence of the piecewise  affine truncated Euler scheme]\label{cvgtruncated} 
Assume the coefficient function $b$ and $\sigma$ satisfy Assumption~(I). Let $(\widetilde X^M_{t})_{t\in [0,T]}$ denote the
piecewise affine   truncated   Euler scheme~\eqref{affineEuler} with step   $h =\frac TM$. 
Then for every $r\in(0, p)$, 
\[
\Big\| \sup_{t\in [0,T]}|\widetilde X^M_{  t} - X_t\big|\Big\|_r\to 0\quad \mbox{ as } \quad h= \tfrac TM \to 0.
\]
\end{prop}

\begin{proof}{\sc Step~1}. 
One checks that, if $t\!\in[t_m, t_{m+1}]$
\begin{align*}
|  X_t-\widetilde X^M_{t} |& \le |  X_t-\widetilde X^M_{t_m} | \vee  |  X_t-\widetilde X^M_{t_{m+1}} |\\
&\le ( |  X_t-X_{t_m}| + |X_{t_m}-\widetilde X^M_{t_m} | \vee(  |  X_t-X_{t_{m+1}}| + |X_{t_{m+1}}-\widetilde X^M_{t_{m+1}} |)\\
& \le \sup_{0\leq m\leq M}\big|X_{t_m}-\widetilde X^M_{t_m}| + w(X,\tfrac TM)
\end{align*}
where $w(x,\delta) = \sup_{|t-s|\le \delta}|x(t)-x(s)|$ is the uniform continuity modulus of a function $x:[0,T]\to \RR$. Hence
\[
\Big\| \sup_{t\in [0,T]} |  X_t-\widetilde X^M_{t} |\Big\|_r  \le  \Big\| \sup_{0\leq m\leq M}\big|X_{t_m}-\widetilde X^M_{t_m}| \Big\|_r + \big\|w(X,\tfrac TM)\big\|_r.
\]
As $X$ has $\PP$-$a.s.$ continuous paths, $w(X,\tfrac TM)\rightarrow0$ $\PP$-$a.s.$ as $M\rightarrow+\infty$. Moreover, it satisfies the domination property $w(X,\tfrac TM)\le 2 \sup_{t\in [0,T]}|X_t| \!\in L^p(\PP)$ by~\eqref{polymoment}, then it follows from Lebesgue's dominated convergence theorem that $ \big\|w(X,\tfrac TM)\big\|_r\to 0$ as $M\to +\infty$.

\noindent {\sc Step~2}. At this stage, it suffices  to prove that
\begin{equation}\label{needconv}
 \Big\| \sup_{0\leq m\leq M}\big|X_{t_m}-\widetilde X^M_{t_m}| \Big\|_r \to 0 \quad \mbox{ as } \quad h= \tfrac TM \to 0.
\end{equation}
Proposition~\ref{convEulerclass}-(c) shows the convergence of the regular Euler scheme \eqref{continuousEulerclass} \[\displaystyle\Big\Vert\sup_{t\in [0, T]}\left| X_{t}-\bar{X}_{t}^{M}\right|\Big\Vert_{p}\rightarrow 0\quad \text{as}\quad h= \tfrac TM \to 0.\]
Consequently, if we exhibit an event   $\Omega_h\!\in {\cal F}$ such that 
\[
\PP(\Omega_h)\to 1\: \text{ as }\: h\to 0 \quad \mbox{ and }\quad \forall\,\omega\!\in \Omega_h,
\quad (\widetilde X^M_{t_m}(\omega))_{m=0,\ldots,M}=(\bar X^M_{t_m}(\omega))_{m=0,\ldots,M}, \; \]
then \eqref{needconv} will follow by noting that, as $X_0\!\in L^{p}(\PP)$ and we have~(\ref{polymot}) and~(\ref{polymomentdiscret}),
\begin{align*}
 \Big\| \sup_{0\leq m\leq M}\big|\widetilde{X}_{t_m}-\bar X^M_{t_m}| \Big\|_r 
&\le  \Big\|  \sup_{s\in [0,T]}\big|\widetilde{X}_s-\bar X^M_{s}|\Big\|_{p}  \PP(\Omega^c_h)^{\frac 1r-\frac{1}{p}}\to 0\quad \text{as}\quad h= \tfrac TM \to 0.
\end{align*}
 
\noindent {\sc Step~3}. Let us consider, as a candidate, the event 
\[
\Omega_h = \bigcap_{m=1}^M  \Big\{|Z_m| \le\frac{1}{2\sqrt{h}\, [\sigma]_{\text{Lip}_x}} \Big\}.
\]
It is clear that on $\Omega_h$, $Z^h_m= Z_m$, $1\leq m \leq M$, by the very definition of $Z^h_m$ in \eqref{Ztruncated}, so that both (discrete time) regular and truncated Euler schemes $\bar X^M_{t_m}$  and $\widetilde X^M_{t_m}$ coincide. It remains to prove that $\PP(\Omega_h)\to 1$ as $h\to 0$. To this end, we will rely on the classical inequality  satisfied by the survival distribution function of the normal distribution
\[
\forall\, z\!\in \RR_+, \quad \PP(Z \ge z) = \int_z^{+\infty} e^{-\frac{u^2}{2}}\frac {du}{\sqrt{2\pi}}\le \frac 12\,e^{-\frac{z^2}{2}}.
\]
The $Z_m, 1\leq m\leq M$ being i.i.d. with normal distribution ${\cal N}(0,1)$ and having in mind that $h=\frac TM$,
\begin{align*}
\PP(\Omega_h)& = \left(1-\PP\Big(|Z|>  \frac{\sqrt{M}}{2\sqrt{T}\, [\sigma]_{\mathrm{Lip}_x}}\Big)\right)^M\ge  \left(1-  e^{-\frac{M}{8T [\sigma]^2_{\mathrm{Lip}_x}}}\right)^M.
\end{align*}

\noindent For $M$ large enough, we have $e^{-\frac{M}{8T [\sigma]^2_{\mathrm{Lip}_x}}}<\frac{1}{2}$. Moreover,  as  $\log (1-u)\ge -2u$ for  $u\!\in (0,1/2]$, then $\PP(\Omega_h)\ge   \exp{\big(-2M  e^{-\frac{M}{8T [\sigma]^2_{\mathrm{Lip}_x}}} \big)}$ for $M$ large enough, which converges to 1 as $M\rightarrow+\infty$. 
This completes the proof.\end{proof}

\subsection{Monotone convex order for the McKean-Vlasov process}

{\color{black}We prove Theorem~\ref{main} in this section. Remark that we will omit the proof of Theorem \ref{sym} as it is very similar to that of Theorem~\ref{main} by applying Proposition \ref{funmonoconvexEuler}-(b). }
 
\begin{proof}[Proof of Theorem~\ref{main}] $(a)$
Let $M\in\mathbb{N}^{*}.$ 
Let $(\widetilde{X}_{t_{m}}^{M})_{m=0, \ldots, M}$ and $(\widetilde{Y}_{t_{m}}^{M})_{m=0, \ldots, M}$ denote random variables defined by the truncated Euler scheme~(\ref{EulertruncatedX}) and~(\ref{EulertruncatedY}). Let $\widetilde{X}^{M}\coloneqq(\widetilde{X}^{M}_{t})_{t\in[0, T]}$, $\widetilde{Y}^{M}\coloneqq(\widetilde{Y}^{M}_{t})_{t\in[0, T]}$ denote its continuous version defined by the piecewise affine construction~(\ref{affineEuler}). It is obvious that $\sup_{t\in[0, T]}\big|\widetilde{X}^{M}_{t}\big|\leq \sup_{0\leq m\leq M}\big|\widetilde{X}_{t_{m}}^{M}\big|$ by the construction so that 
\begin{align}\label{supxy}
&\left\Vert\sup_{t\in[0, T]}\Big|\widetilde{X}^{M}_{t}\Big|\right\Vert_{p}\leq \left\|  \sup_{0\leq m\leq M}|\widetilde X^M_{t_m}|\right\|_p\leq C\big (1+\|X_0\|_p\big)<+\infty\nonumber\\
&\left\Vert\sup_{t\in[0, T]}\Big|\widetilde{Y}^{M}_{t}\Big|\right\Vert_{p}\leq \left\|  \sup_{0\leq m\leq M}|\widetilde Y^M_{t_m}|\right\|_p\leq C\big (1+\|Y_0\|_p\big)<+\infty
\end{align}
as $X_{0}, Y_{0}\in L^{p}(\mathbb{P})$. 
Hence, $F(X)$ and $F(Y)$ are in $L^{1}(\mathbb{P})$ since $F$ has a $r$-polynomial growth, $r<p$.

We define a function $F_{M}: \RR^{M+1}\rightarrow \RR$ by 
\begin{equation}
x_{0:M}\in\RR^{M+1}\mapsto F_{M}(x_{0:M})\coloneqq F\big(i_{M}(x_{0:M})\big).
\end{equation}
The function $F_{M}$ is obviously convex since $i_{M}$ is a linear application. Moreover, $F_{M}$ has also an $r$-polynomial growth on $\RR^{M+1}$ in the sense of~(\ref{rpolygrowth}) and is non-decreasing with respect to the partial order~(\ref{partialorder}) by the assumption made on $F$. 

It follows from Proposition~\ref{funmonoconvexEuler} that 
\begin{align}\label{im}
\EE &F\big(\widetilde{X}^{M}\big)=\EE F\big(i_{M}(\widetilde{X}_{0}^{M}, \ldots, \widetilde{X}_{M}^{M})\big)=\EE F_{M}\big(\widetilde{X}_{0}^{M}, \ldots, \widetilde{X}_{M}^{M}\big)\nonumber\\
&\hspace{1cm}\leq \EE F_{M}\big(\widetilde{Y}_{0}^{M}, \ldots, \widetilde{Y}_{M}^{M}\big)=\EE F\big(i_{M}(\widetilde{Y}_{0}^{M}, \ldots, \widetilde{Y}_{M}^{M})\big)=\EE F\big(\widetilde{Y}^{M}\big).
\end{align}
As Proposition~\ref{cvgtruncated} implies that $\widetilde{X}$, $\widetilde{Y}$ weakly converges to $X$, $Y$, 
 the inequality~(\ref{im}) implies that 
\[\EE F(X)\leq \EE F(Y),\]
{\color{black}by letting $M\rightarrow +\infty$ as the random variable sequences $\big(F(\widetilde{X}^{M})\big)_{M\geq 1}$, $\big(F(\widetilde{Y}^{M})\big)_{M\geq 1}$ are uniformly integrable by~(\ref{supxy}) since $F$ has an $r$-polynomial growth.}

\noindent $(b)$ The proof of Part $(b)$ follows the same idea as $(a)$.

By the same idea as in Section~\ref{funcmcvEuler}, we consider a function  
\[\tilde{G}: (x_{0:M},\eta_{0:M})\in\RR^{M+1}\times \big(\mathcal{P}_{p}(\RR)\big)^{M+1}\:\longmapsto \:\tilde{G}(x_{0:M},\eta_{0:M})\in \RR\]
satisfying the following conditions $(i)^{G}$, $(ii)^{G}$ and $(iii)^{G}$:\vspace{-0.4cm}
\begin{enumerate}[$(i)^{G}$]
\item $\tilde{G}$ is convex and non-decreasing with respect to the partial order~\ref{partialorder} in $x_{0:M}$;
\item $\tilde{G}$ is non-decreasing in $\mu_{0:M}$ with respect to the monotone convex order in the sense that 
\begin{flalign}
&\forall x_{0:M}\in\RR^{M+1} \text{ and } \forall\mu_{0:M}, \nu_{0:M}\in\big(\mathcal{P}_{p}(\RR)\big)^{M+1} \text{ s.t. }\mu_{i}\mconright\nu_{i},\; 0\leq i \leq M,&\nonumber\\
&\hspace{3.5cm}\tilde{G}(x_{0:M}, \mu_{0:M})\leq\tilde{G}(x_{0:M}, \nu_{0:M});&\nonumber
\end{flalign}
\item $\tilde{G}$ has an $r$-polynomial growth, $1\leq r \leq p$, in the sense that 
\begin{flalign}
&\exists\,  C\in \mathbb{R}_{+}\text{ s.t. } \forall (x_{0:M}, \mu_{0:M})\in\RR^{M+1}\times \big(\mathcal{P}_{p}(\RR)\big)^{M+1},&\nonumber\\
&\hspace{3cm}\tilde{G}(x_{0:M}, \mu_{0:M})\leq C \big[1+\sup_{0\leq m\leq M}\left|x_{m}\right|^{r}+\sup_{0\leq m\leq M}\mathcal{W}_{p}^{r}(\mu_{m}, \delta_{0})\big].&
\end{flalign}
\end{enumerate}

By applying the same idea as Proposition~\ref{funmonoconvexEuler} and considering the following \[\Phi_{m}^{\tilde{G}}, \Psi_{m}^{\tilde{G}}: \RR^{m+1}\times\big(\mathcal{P}_{p}(\RR)\big)^{M+1}\rightarrow \RR, \;m=0, \ldots, M,\]  defined by  
\begin{align}
&\Phi^{\tilde{G}}_{M}(x_{0:M}, \mu_{0:M})=\tilde{G}(x_{0:M}, \mu_{0:M}),\nonumber\\
& \Phi^{\tilde{G}}_{m}(x_{0:m}, \mu_{0:M})=\big(P^{h}_{m+1}\Phi^{\tilde{G}}_{m+1}(x_{0:m}, \;\cdot\;, \mu_{0:M})\big)\big(x_{m}, \mu_{m}\big),\nonumber \\
 &\hspace{2.4cm}=\EE \Big[\Phi^{\tilde{G}}_{m+1}\big(x_{0:m}, x_{m}+hb(t_m, x_m, \mu_m)+\sqrt{h}\sigma(t_m, x_{m}, \mu_{m})Z^{h}_{m+1}, \mu_{0:M}\big)\Big],\nonumber\\
&\Psi^{\tilde{G}}_{M}(x_{0:M}, \mu_{0:M})=\tilde{G}(x_{0:M}, \mu_{0:M}),\nonumber\\
 &\Psi^{\tilde{G}}_{m}(x_{0:m}, \mu_{0:M})=
\EE \Big[\Psi^{\tilde{G}}_{m+1}\big(x_{0:m}, x_{m}+h\beta(t_m, x_m, \mu_m)+\sqrt{h}\theta(t_m, x_{m}, \mu_{m})Z^{h}_{m+1}, \mu_{0:M}\big)\Big], \nonumber
\end{align}
in the place of $\Phi_{m}$ and $\Psi_{m}$ in~(\ref{defPhi1}),~(\ref{defPhi2}) and~(\ref{defpsi}), one can prove that 
\begin{equation}\label{convgeuler}
\EE \tilde{G}(\widetilde{X}_{t_0}, \ldots, \widetilde{X}_{t_M}, \widetilde{\mu}_{t_0}, \ldots, \widetilde{\mu}_{t_M})\leq \EE \tilde{G}(\widetilde{Y}_{t_0}, \ldots, \bar{Y}_{t_M}, \widetilde{\nu}_{t_0}, \ldots, \widetilde{\nu}_{t_M}).
\end{equation}

Moreover, for every $t\in[0, T]$, let $(\widetilde{\mu}_{t})_{t\in[0,T]}$ denote the marginal probability distribution of the process $(\widetilde{X}_{t})_{t\in[0, T]}$ defined in \eqref{affineEuler}. Then Proposition~\ref{cvgtruncated} implies that 
\begin{equation}\label{contimu}
d_{\,r}\big((\widetilde{\mu}_{t})_{t\in[0, T]}, (\mu_{t})_{t\in[0, T]}\big)=\sup_{t\in[0, T]}\mathcal{W}_{r}(\widetilde{\mu}_{t}, \mu_{t})\leq \vertii{\sup_{t\in[0, T]}\big|\widetilde{X}_{t}^{M}-X_{t}\big|}_{r}\rightarrow 0 \quad \text{as}\quad h=\frac{T}{M}\rightarrow0.
\end{equation}

Remark that if we generalize the interpoler $i_{M}$ in Definition~\ref{interpolator} to the marginal distribution space $\Big(\mathcal{C}\big([0, T], \mathcal{P}_{r}(\RR)\big), \;d_{r}\Big)$ defined in (\ref{margspace}) and (\ref{disdc}) as follows 
\begin{align}
 &\forall\, m=0, \ldots, M-1, \;\forall\, t\in[t_{m}^{M}, t_{m+1}^{M}], \nonumber\\
&\forall \, \mu_{0:M}\in\big(\mathcal{P}_{p}(\RR)\big)^{M+1},\hspace{1.25cm}i_{M}(\mu_{0:M})(t)=\frac{M}{T}\big[(t_{m+1}^{M}-t)\mu_{m}+(t-t^{M}_{m})\mu_{m+1}\big],\nonumber
\end{align}
where for every Borel set $A$ and $\lambda \in [0, 1]$, $(\lambda \mu_{t_{m}}+(1-\lambda)\mu_{t_{m+1}})(A)\coloneqq\lambda \mu_{t_{m}}(A)+(1-\lambda)\mu_{t_{m+1}}(A),$ then \[(\widetilde{\mu}_{t})_{t\in [0, T]}=i_{M}(\widetilde{\mu}_{t_{0}},\ldots, \widetilde{\mu}_{t_{M}}).\]
Thus, for every $(x_{0:M}, \eta_{0:M})\in \RR^{M+1}\times \big(\mathcal{P}(\RR)\big)^{M+1}$, we define $G_{M}(x_{0:M}, \eta_{0:M})\coloneqq G\big(i_{M}(x_{0:M}), i_{M}(\eta_{0:M})\big)$. Then by the hypotheses made on $G$, $G_{M}$ satisfies the previous conditions $(i)^{G}$, $(ii)^{G}$ and $(iii)^{G}$. Consequantly, for every $M\in \mathbb{N}^{*}$,  we have
\begin{align}
\EE G&\big(\widetilde{X}^{M}, (\widetilde{\mu}^{M}_{t})_{t\in[0, T]}\big)=\EE G\big(i_{M}(\widetilde{X}_{t_{0}}^{M}, \ldots, \widetilde{X}_{t_{M}}^{M}), i_{M}(\widetilde{\mu}_{t_{0}}^{M}, \ldots, \widetilde{\mu}_{t_{M}}^{M})\big)\nonumber\\
&=\EE G_{M}\big(\widetilde{X}_{t_{0}}^{M}, \ldots, \widetilde{X}_{t_{M}}^{M}, \widetilde{\mu}_{t_0}^{M}, \ldots, \widetilde{\mu}_{t_{M}}^{M}\big)\leq \EE G_{M}\big(\widetilde{Y}_{t_{0}}^{M}, \ldots, \widetilde{Y}_{t_{M}}^{M}, \widetilde{\nu}_{t_{0}}^{M}, \ldots, \widetilde{\nu}_{t_{M}}^{M}\big)\quad\text{(by \eqref{convgeuler})}\nonumber\\
&=\EE G\big(i_{M}(\widetilde{Y}_{t_{0}}^{M}, \ldots, \widetilde{Y}_{t_{M}}^{M}), i_{M}(\widetilde{\nu}_{t_{0}}^{M}, \ldots, \widetilde{\nu}_{t_{M}}^{M})\big)=\EE G\big(\widetilde{Y}^{M}, (\widetilde{\nu}_{t}^{M})_{t\in[0, T]}\big).
\end{align}

Then one can obtain~(\ref{result2}) {\color{black}by letting $M\rightarrow +\infty$ as  the random variables $\Big(G\big(\widetilde{X}^{M}, (\widetilde{\mu}^{M}_{t})_{t\in[0, T]}\big)\Big)_{M\geq 1}$ and $\Big(G\big(\widetilde{Y}^{M}, (\widetilde{\nu}^{M}_{t})_{t\in[0, T]}\big)\Big)_{M\geq 1}$ are respectively uniformly integrable and  $G$ is continuous with respect to the distance $d_{\,r}$} owing to the hypotheses made on $G$ (see Theorem~\ref{main}-$(b)$-$(ii)$ and $(iii)$) and~(\ref{supxy}). 
\end{proof}

\section*{\large Appendix A. Proof of Lemma~\ref{onlylineargrowth}}




\begin{proof}[Proof of Lemma~\ref{onlylineargrowth}]
By the definition of the monotone convex order~(\ref{defmcv2}), we only need to prove that if for every convex, non-decreasing function $g$ with linear growth, $\int_{\RR} g d\mu \le \int_{\RR} g d\nu$, then $\mu\mconright \nu$. 


Let $f : \RR\rightarrow \RR$ be a convex, non-decreasing function. 
If its right derivative $f_{r}'=0$, it is trivial that $\int_{\RR} f d\mu \le \int_{\RR} f d\nu$ as $f$ has linear growth. Consider now there exists at least one point $x_0$ such that $f_{r}'(x_0)>0$. 
Set for every  $K>0$
\[
f_{_{\!K}}(x) = f(x)\mbox{\bf 1}_{\{x\le K\}} + \big(f(K)+ f'_r(K)(x-K)\big)\mbox{\bf 1}_{\{x\ge K\}},
\]
where we replace $f$ by its right tangent on $[K, +\infty)$.
The function $f_{_{\!K}}$ is convex, non-decreasing since $f'_r(K) \ge 0$ by monotony of $f$. Moreover $f_{_{\!K}}$ is Lipschitz continuous  since $0\le (f_{_{\!K}})'_r(x)\le f'_r(K)$, hence $f_{_{\!K}}$ has linear growth. Hence,  for every  $K>0$
\begin{equation}\label{eq:mcvxK}
 \int f_{_{\!K}} d \mu \leq \int f_{_{\!K}} d\nu
 \end{equation}
and naturally 
 \begin{equation}\label{eq:mcvxK}
 \int f_{_{\!K}}^{+} d \mu +  \int f_{_{\!K}}^{-} d \mu \leq \int f_{_{\!K}}^{+} d\nu + \int f_{_{\!K}}^{-} d\nu.
 \end{equation}
As  $f_{r}'(x_0)>0$ for some $x_0$, then $f(x) \to +\infty$ as $x \to + \infty$ hence $f(x) \ge 0$ for  $x$ larger than some $K_0\ge 0$. Consequently, the family $(f_{K}^{+})_{K\geq K_0}$ is non-decreasing in $K$ and  $(f_{_{\!K}})^- = f^-$ for any $K\ge K_0$. Hence 
\[\lim_{K\rightarrow +\infty} \int f_{_{\!K}}^{+} d \mu=  \int f^{+} d \mu\quad \text{and}\quad \lim_{K\rightarrow +\infty} \int f_{_{\!K}}^{+} d \nu=  \int f^{+} d \nu\]
by applying Beppo-Levi's  monotone  convergence theorem. Then we have $\int_{\RR} f d\mu \le \int_{\RR} f d\nu$ for any convex and non-decreasing function $f$, which implies that $\mu\mconright\nu$. 
\end{proof}

\section*{\large \color{black}Appendix B. Proof of Proposition~\ref{ifnondecresing} and counterexample}

\begin{proof}[Proof of Proposition~\ref{ifnondecresing}]
As $\mu\in\mathcal{P}_{1}(\RR)$ is fixed, to alleviate notations, we denote $b_m(x)= b(t_m,x,\mu)$ and $\sigma_m(x)= \sigma(t_m,x,\mu)$, $0\leq m\leq M$. By Assumption (I), $b_m(x)$ and $\sigma_m(x)$ are Lipschitz continuous functions with respective Lipschitz constants $[b_m]_{\Lip}$ and $[\sigma_m]_{\Lip}$ satisfying $[b_m]_{\Lip}\leq L$ and $[\sigma_m]_{\Lip}\leq L$. For every $0\leq m\leq M-1$, let $P_{m+1}^{h}$ denote the transition operator of the regular Euler scheme defined by
\[(P_{m+1}^{h}f)(x)\coloneqq \EE f\big(\mathcal{E}_{m}^{h}(x, \mu, Z_{m+1})\big)=\EE f\big(x+h\cdot b_{m}(x)+\sqrt{h}\cdot \sigma_{m}(x)Z_{m+1}\big).\]

The convexity of $x\mapsto (P_{m+1}^{h}f)(x)$ is obvious by Assumption (II) and the convexity and monotony of $f$. Now we prove that the function $x\mapsto (P_{m+1}^{h}f)(x)$ is non-decreasing. First, remark that for every $x\in \RR$,
\[\EE f\big(x+h\cdot b_{m}(x)+\sqrt{h}\cdot \sigma_{m}(x)Z_{m+1}\big)=\EE f\big(x+h\cdot b_{m}(x)+\sqrt{h}\cdot \left|\sigma_{m}(x)\right|Z_{m+1}\big)\]
as $Z_{m+1}$ has a symmetric distribution $\mathcal{N}(0, 1)$, so in what follows we assume that $\sigma\geq0$ (otherwise, we only need to consider $\left|\sigma\right|$ instead of $\sigma$). 

\smallskip

\smallskip
\noindent {\sc Step~1} ($f$ {\em smooth}). Assume $f$ is also $C^1$ and   both $f$ and $f'$  have  sub-exponential growth.  Then $f'\ge 0$ as $f$ is non-decreasing. Let $x$, $y\!\in \RR$, $x> y$. A first order Taylor expansion yields
\begin{align}
&\mathbb{E}\, f\big(\EMH(x, \mu, Z_{m+1})\big)-\mathbb{E}\, f\big(\EMH(y, \mu, Z_{m+1})\big)\nonumber\\
&=\EE\Big[\int_0^1f'\big(u\, {\cal E}_m^{h}(x,\mu,Z_{m+1})+(1-u){\cal E}_m^{h}(y,\mu,Z_{m+1})\big)du\nonumber\\
&\hspace{2cm} \cdot\big(x-y +h\big(b_m(x)-b_m(y)\big) +\sqrt{h}\big( \sigma_m(x)-\sigma_m(y)\big)Z_{m+1} \big)  \Big].\nonumber
\end{align}

The coefficient function $\sigma$ is assumed to be non-decreasing in $x$ so that $\sigma_m(x)-\sigma_m(y)\geq 0$. As $f$ is convex, $f'$ is non-decreasing so that the function $\displaystyle z\mapsto \int_0^1f'\big(u\, {\cal E}_m^{h}(x,\mu,z)+(1-u){\cal E}_m^{h}(y,\mu,z)\big)du$ and the function $\displaystyle z\mapsto (x-y +h\big(b_m(x)-b_m(y)\big) +\sqrt{h}\big( \sigma_m(x)-\sigma_m(y)\big)z$ are both non-decreasing. Hence, 
\begin{align}\label{comonotony}
&\mathbb{E}\, f\big(\EMH(x, \mu, Z_{m+1})\big)-\mathbb{E}\, f\big(\EMH(y, \mu, Z_{m+1})\big)\nonumber\\
&=\EE\Big[\int_0^1f'\big(u\, {\cal E}_m^{h}(x,\mu,Z_{m+1})+(1-u){\cal E}_m^{h}(y,\mu,Z_{m+1})\big)du\nonumber\\
&\hspace{2cm} \cdot\big(x-y +h\big(b_m(x)-b_m(y)\big) +\sqrt{h}\big( \sigma_m(x)-\sigma_m(y)\big)Z_{m+1} \big)  \Big]\nonumber\\
&\geq\EE\Big[\int_0^1f'\big(u\, {\cal E}_m^{h}(x,\mu,Z_{m+1})+(1-u){\cal E}_m^{h}(y,\mu,Z_{m+1})\big)du\Big]\nonumber\\
&\hspace{2cm} \cdot\EE\Big[x-y +h\big(b_m(x)-b_m(y)\big) +\sqrt{h}\big( \sigma_m(x)-\sigma_m(y)\big)Z_{m+1}   \Big]\\
&=\EE\Big[\int_0^1f'\big(u\, {\cal E}_m^{h}(x,\mu,Z_{m+1})+(1-u){\cal E}_m^{h}(y,\mu,Z_{m+1})\big)du\Big]\cdot\Big(x-y +h\big(b_m(x)-b_m(y)\big) \Big)\geq 0.\nonumber
\end{align}


\noindent {\sc Step~2} ({\em Regularization}). If $f$ is simply non-decreasing, one can consider the function $f_{\varepsilon}(x)\coloneqq \EE f(x+\sqrt{\varepsilon}\zeta)$ with $\varepsilon>0$ and $\zeta\sim \mathcal{N}(0,1)$ independent to $(Z_{1},\ldots , Z_{M})$. Then by the same argument as in the proof of Lemma~\ref{propamonoconv}-$(a)$, the function $f_{\varepsilon}$ is $C^{1}$, both $f_{\varepsilon}$ and $f_{\varepsilon}'$  have sub-exponential growth and $\EE f_{\varepsilon}\big(\mathcal{E}_{m}^{h}(x,\mu,Z_{m+1})\big)$ converges to $\EE f\big(\mathcal{E}_{m}^{h}(x,\mu,Z_{m+1})\big)$ as $\varepsilon\rightarrow0$. Hence, $f_{\varepsilon}$ is non-decreasing owing to Step 1 and one concludes by letting $\varepsilon\rightarrow0$. \end{proof}





The classical co-monotony argument in \eqref{comonotony} suggests the following counter-example for a decreasing $\sigma$. 
Set $f(x)= e^x, \; b(t,x)= 0$ and 
$ \sigma(t,x, \mu)=\sigma(x) :=\EE\, (\zeta-x)^+$ with $\zeta \sim {\cal N}(0,1).$ 
The function $\sigma$ is convex, decreasing and $\sigma'(x) = -\PP(\zeta>x) <0$. Elementary computations yield 
\[
(P_m^{h} f)'(x)= e^{x+h\frac{\sigma^2(x)}{2}}\Big( 1+\tfrac h2 \sigma\sigma'(x)\Big).
\] 
One checks that $\sigma(x)\sim -x$ and $ \sigma'(x)=-1$   as $x\to -\infty$  so that $\sigma\sigma'(x) \sim x$ as $x\to -\infty$. Consequently, $h$ being fixed, the above derivative becomes negative as $x\to-\infty$.



\bibliographystyle{alpha}
\bibliography{monotone_convex_order}

\end{document}